\documentclass[10pt,a4paper]{article}

\usepackage{latexsym,amsfonts,amsmath,amssymb,mathrsfs,url,color,amsthm}
\usepackage{graphicx}

\newtheorem{theorem}{Theorem}
\newtheorem{lemma}{Lemma}
\newtheorem{algorithm}{Algorithm}
\newtheorem{remark}{Remark}

\newtheorem{definition}{Definition}
\newtheorem{corollary}{Corollary}

\usepackage[dvipdfm,
linkbordercolor={1 0 0},
citebordercolor={0 1 0},
urlbordercolor={0 1 1}]{hyperref}

\newcommand{\rd}{\,\mathrm{d}}
\newcommand{\rtr}{\,\mathrm{tr}}
\newcommand{\bsc}{\boldsymbol{c}}
\newcommand{\bsx}{\boldsymbol{x}}
\newcommand{\bsy}{\boldsymbol{y}}
\newcommand{\bsz}{\boldsymbol{z}}
\newcommand{\bsk}{\boldsymbol{k}}
\newcommand{\bsl}{\boldsymbol{l}}
\newcommand{\bsr}{\boldsymbol{r}}

\newcommand{\bsq}{\boldsymbol{q}}
\newcommand{\bsalpha}{\boldsymbol{\alpha}}

\newcommand{\bseta}{\boldsymbol{\eta}}
\newcommand{\bsgamma}{\boldsymbol{\gamma}}
\newcommand{\bssigma}{\boldsymbol{\sigma}}

\newcommand{\bszero}{\boldsymbol{0}}

\newcommand{\nat}{\mathbb{N}}

\newcommand{\RR}{\mathbb{R}}

\newcommand{\FF}{\mathbb{F}}

\newcommand{\Dcal}{\mathcal{D}}
\newcommand{\Ecal}{\mathcal{E}}

\newcommand{\Lcal}{\mathcal{L}}

\newcommand{\wal}{\mathrm{wal}}

\allowdisplaybreaks

\begin{document}

\title{Fast construction of higher order digital nets for numerical integration in weighted Sobolev spaces\thanks{This work was supported by Grant-in-Aid for JSPS Fellows No.24-4020.}}

\author{Takashi Goda\thanks{Graduate School of Engineering, The University of Tokyo, 7-3-1 Hongo, Bunkyo-ku, Tokyo 113-8656 ({\tt goda@frcer.t.u-tokyo.ac.jp}).}}

\date{\today}
\maketitle

\begin{abstract}
Higher order digital nets are special classes of point sets for quasi-Monte Carlo rules which achieve the optimal convergence rate for numerical integration of smooth functions. An explicit construction of higher order digital nets was proposed by Dick, which is based on digitally interlacing in a certain way the components of classical digital nets whose number of components is a multiple $ds$ of the dimension $s$. In this paper we give a fast computer search algorithm to find good classical digital nets suitable for interlaced components by using polynomial lattice point sets.

We consider certain weighted Sobolev spaces of smoothness of arbitrarily high order, and derive an upper bound on the mean square worst-case error for digitally shifted higher order digital nets. Employing this upper bound as a quality criterion, we prove that the component-by-component construction can be used efficiently to find good polynomial lattice point sets suitable for interlaced components. Through this approach we are able to get some tractability results under certain conditions on the weights. Fast construction using the fast Fourier transform requires the construction cost of $O(dsN \log N)$ operations using $O(N)$ memory, where $N$ is the number of points and $s$ is the dimension. This implies a significant reduction in the construction cost as compared to higher order polynomial lattice point sets. Numerical experiments confirm that the performance of our constructed point sets often outperforms those of higher order digital nets with Sobol' sequences and Niederreiter-Xing sequences used for interlaced components, indicating the usefulness of our algorithm.
\end{abstract}

\section{Introduction}
\label{intro}

In this paper we study multivariate integration of smooth functions defined over the $s$-dimensional unit cube $[0,1)^s$,
  \begin{align*}
    I(f) =  \int_{[0,1)^s}f(\bsx)\rd\bsx .
  \end{align*}
Quasi-Monte Carlo (QMC) rules approximate $I(f)$ by
  \begin{align*}
    Q(f;P_N) = \frac{1}{N}\sum_{n=0}^{N-1}f(\bsx_n) ,
  \end{align*}
where $P_N=\{\bsx_0,\ldots, \bsx_{N-1}\}\subset [0,1)^s$ is a carefully chosen point set consisting of $N$ points. Here $P_N$ is understood as a multiset, i.e., a set in which the multiplicity of elements matters. We refer to \cite{DKS13,Lem09} for general information on QMC rules. The Koksma-Hlawka inequality ensures that QMC rules typically achieve an integration error of $O(N^{-1+\delta})$ with arbitrarily small $\delta>0$ when the integrand has a bounded variation in the sense of Hardy and Krause. Two prominent ways to generate good point sets are integration lattices, see for example \cite{Nie92a,SJ94}, and digital nets, see for example \cite{DP10,Nie92a}. 

Digital nets in a prime base $b$ are defined as follows. Let $\FF_b:=\{0,\ldots,b-1\}$ be the finite field with $b$ elements equipped with addition modulo $b$ and multiplication modulo $b$. For each point $\bsx_n=(x_{n,1},\ldots,x_{n,s})\in P_N$, we denote the $b$-adic expansion of $x_{n,j}$ by
  \begin{align*}
    x_{n,j} = \frac{\xi_{n,j,1}}{b}+\frac{\xi_{n,j,2}}{b^2}+\cdots ,
  \end{align*}
where $\xi_{n,j,1},\xi_{n,j,2},\ldots \in \FF_b$ for $1\le j\le s$, which is unique in the sense that infinitely many of the $\xi_{n,j,i}$ are different from $b-1$. Through this expansion, each point $\bsx_n$ can be identified with one element in $\FF_b^{s\times \infty}$, and thus, a point set $P_N$ can be identified with a subset of $\FF_b^{s\times \infty}$. Then $P_N$ is called a digital net in base $b$ if it can be identified with an {\em $\FF_b$-linear subspace} of $\FF_b^{s\times \infty}$ with possible multiplicity. If we add one condition that $\xi_{n,j,m+1}=\xi_{n,j,m+2}=\ldots=0$, that is, every $x_{n,j}$ must be represented by $m$-digit $b$-adic fraction for a positive integer $m$ such that $N=b^m$, this definition is in accordance with the standard definition of digital nets based on generating matrices as introduced by Niederreiter \cite{Nie92a}. Thus, we consider a more general class of point sets as digital nets in this paper.

The quality of a digital net is often measured by the so-called $t$-value, which is given as follows. Let $P_{b^m}$ be a digital net in base $b$. If every elementary interval of the form
  \begin{align*}
    \prod_{j=1}^{s}\left[ \frac{a_j}{b^{d_j}},\frac{a_j+1}{b^{d_j}}\right) ,
  \end{align*}
for every choice of non-negative integers $a_1,\ldots,a_s$ and $d_1,\ldots,d_s$ such that $0\le a_j<b^{d_j}$ and $d_1+\cdots+d_s=m-t$, contains exactly $b^t$ points, we call $P_{b^m}$ a digital $(t,m,s)$-net in base $b$. Obviously, a digital net with a smaller $t$-value has better distribution properties.

Explicit constructions of digital $(t,m,s)$-nets with small $t$-value have been proposed by many researchers, such as Sobol' \cite{Sob67}, Faure \cite{Fau82}, Niederreiter \cite{Nie88} and Niederreiter and Xing \cite{NX01} as well as others, see \cite[Section~8]{DP10} for more information. Polynomial lattice point sets, first proposed in \cite{Nie92b}, are a special construction of digital nets based on rational functions over $\FF_b$. QMC rules using polynomial lattice point sets are called polynomial lattice rules. In order to construct good polynomial lattice point sets, we need to find good polynomials over $\FF_b$. Seen in this light, polynomial lattice point sets are not fully explicit, but provide us with a flexibility in adjusting point sets to a given problem at hand. In order to make a clear distinction from higher order digital nets that shall be introduced below, we use the word {\em classical digital nets} as a term referring to digital $(t,m,s)$-nets and polynomial lattice point sets throughout this paper.

In order to achieve an improved convergence rate, there have been several attempts to exploit some smoothness of integrands in QMC rules. In \cite{Owe97}, Owen regards a function as smooth if its order $s$ mixed partial derivative satisfies a H\"{o}lder condition of order $0<\beta\le 1$, and for this concept of smoothness, he proved that scrambled classical digital nets can achieve an improved convergence of the root mean square error of $O(N^{-3/2+\delta})$. As a different concept of smoothness, by regarding a function as smooth if its partial mixed derivatives up to order $\alpha\ge 2$, $\alpha \in \nat$, in each variable are square integrable, it is possible to achieve a further improved convergence rate, as done in \cite{BD09,Dic07,Dic08,Dic11,Gxx1,Gxx2,GDxx} and also in this paper. Therefore, throughout this paper, smoothness of a function is understood in the latter sense, where an integer $\alpha\ge 2$ represents the smoothness, and a function of smoothness $\alpha$ is called an $\alpha$-smooth function.

In order for QMC rules to achieve the optimal order of convergence for $\alpha$-smooth functions, special classes of digital nets were proposed by Dick in \cite{Dic07,Dic08} based on analyzing the decay of the Walsh coefficients of $\alpha$-smooth functions. Since QMC rules based on these special digital nets achieve higher order convergence than the typical convergence $O(N^{-1+\delta})$, say $O(N^{-\alpha+\delta})$, we call these special digital nets {\em higher order digital nets}. Higher order digital nets in a reproducing kernel Hilbert space consisting of $\alpha$-smooth periodic functions were first studied in \cite{Dic07}, the result in which was later extended in \cite{Dic08} to the case in a normed Walsh space consisting of $\alpha$-smooth non-periodic functions. It was shown in \cite{BD09} that the result from \cite{Dic08} can be also achieved in a reproducing kernel Hilbert space that is different from what is studied in \cite{Dic07}. More specifically, a weighted unanchored Sobolev space of order $\alpha \ge 2$, $\alpha\in \nat$, was considered as a reproducing kernel Hilbert space and it was proven that QMC rules using higher order digital nets achieve the optimal order of convergence not only of the worst-case error but also of the mean square worst-case error with respect to a random digital shift. More recently, QMC rules using higher order scrambled digital nets were studied in \cite{Dic11,GDxx} to achieve the optimal rate of the root mean square error for $\alpha$-smooth functions.

So far, there are two algorithms for constructing higher order digital nets. One is given in \cite{DP07} by generalizing the definition of polynomial lattice point sets. These nets are called higher order polynomial lattice point sets, and QMC rules using higher order polynomial lattice point sets are called higher order polynomial lattice rules. Regarding a computer search algorithm for finding higher order polynomial lattice rules which achieve the optimal rate of convergence, the component-by-component (CBC) construction is studied in \cite{BDGP11,BDLNP12}. Even with efficient calculation of the worst-case error as given in \cite{BDLNP12}, however, a construction cost of $O(\alpha sN^{\alpha } \log N)$ operations using $O(N^{\alpha})$ memory is required. Recently, in \cite{Gxx1}, the author achieved a cost reduction to $O(\alpha sN^{\alpha /2} \log N)$ operations using $O(N^{\alpha/2})$ memory by considering higher order polynomial lattice point sets over $\FF_2$ which are randomized by a digital shift and then folded by the baker's transformation \cite{Hi02}. Nevertheless, the exponential dependence on $\alpha$ of the construction cost degrades the availability of these rules as $\alpha$ increases.

The other algorithm for constructing higher order digital nets, proposed in \cite{Dic07,Dic08}, applies a digit interlacing function to the components of classical digital nets whose number of components is a multiple of the dimension. Let $\{\bsy_0,\ldots,\bsy_{b^m-1}\}$ be a digital $(t',m,ds)$-net in base $b$ for $d\in \nat$. For each point $\bsy_n=(y_{n,1},\ldots, y_{n,ds})$, we denote the $b$-adic expansion of $y_{n,j}$ by
  \begin{align*}
    y_{n,j}=\frac{\eta_{n,j,1}}{b}+\frac{\eta_{n,j,2}}{b^2}+\cdots ,
  \end{align*}
where $\eta_{n,j,1},\eta_{n,j,2},\ldots \in \FF_b$ for $1\le j\le ds$. Then a higher order digital net $\{\bsx_0,\ldots,\bsx_{b^m-1}\}$ is given as follows: Each coordinate of $\bsx_n=(x_{n,1},\ldots, x_{n,s})$ is obtained by interlacing the components $y_{n,d(j-1)+1},\ldots,y_{n,dj}$ digitally as
  \begin{align*}
    x_{n,j} = \sum_{a=1}^{\infty}\sum_{r=1}^{d}\frac{\eta_{n,d(j-1)+r,a}}{b^{r+(a-1)d}} ,
  \end{align*}
for $1\le j\le s$. Dick \cite{Dic08} proved that QMC rules using $\{\bsx_0,\ldots,\bsx_{b^m-1}\}$ can achieve a worst-case error of order $N^{-\min(\alpha,d)+\delta}$ for any $\delta>0$.

Hence, the major advantage of the latter algorithm is that we can use existing digital $(t,m,s)$-nets. The disadvantage, on the other hand, is that the $t$-value of digital $(t,m,s)$-nets increases at least linearly in $s$, so that it becomes hard to obtain a bound on the worst-case error independent of the dimension. This observation motivates us to replace existing digital $(t',m,ds)$-nets by polynomial lattice point sets in dimension $ds$ that are suitably chosen for interlaced components. For this purpose, however, we need to find good polynomials over $\FF_b$ by using some computer search algorithm.

In a similar context, there exists a result in \cite{GDxx} where scrambled polynomial lattice point sets are used as interlaced components to construct higher order scrambled digital nets. It was shown there that we are able to obtain a good dependence on the dimension of the root mean square error. Moreover, as compared to higher order polynomial lattice rules, the computational cost for the fast CBC construction could be significantly reduced to $O(dsN \log N)$ operations using $O(N)$ memory. Thus, as a further study of \cite{GDxx}, it is worth investigating whether good polynomial lattice point sets can be efficiently obtained for interlaced components to achieve the optimal rate either of the worst-case error or the mean square worst-case error with respect to some randomization, while obtaining a good dependence on the dimension and reducing the construction cost as compared to higher order polynomial lattice point sets. This is exactly what we are interested in here.

In this paper, we consider weighted unanchored Sobolev spaces of order $\alpha$ as studied in \cite{BD09} and derive a computable upper bound on the mean square worst-case error for digitally shifted higher order digital nets. Employing this upper bound as a quality criterion, we prove that the CBC construction requires a construction cost of $O(dsN \log N)$ operations using $O(N)$ memory to find good polynomial lattice point sets that are used for interlaced components. Thus, our obtained construction cost is the same as that in \cite{GDxx} and is much lower than those in \cite{BDGP11,BDLNP12,Gxx1}.  The main difference of this study from \cite{GDxx} is twofold: One is that instead of scrambling we consider a randomization of point sets by using a digital shift that is computationally much cheaper to implement. The other is that we employ the mean square worst-case error instead of the mean square error as an error criterion. We note here that a small mean square worst-case error implies the existence of a digitally shifted point set which yields a reasonably small worst-case error. Numerical experiments in Subsection \ref{subsec:test} show that randomization by a digital shift can be used in place of scrambling as a cheap way to obtain some statistical estimate on the integration error. As a continuation of this paper, we shall study construction algorithms of deterministic higher order digital nets by using polynomial lattice point sets for interlaced components in another paper \cite{Gxx2}.

The remainder of this paper is organized as follows. In the next section, we introduce the necessary background and notation including Walsh functions, polynomial lattice rules, higher order digital nets and their randomization, and weighted unanchored Sobolev spaces of order $\alpha \ge 2$. In Section \ref{bound}, we study the mean square worst-case error for digitally shifted higher order digital nets with the aim to derive a computable upper bound on the error. We show in Section \ref{cbc} that the CBC construction can be used to obtain good polynomial lattice point sets as interlaced components. QMC rules using digitally shifted higher order digital nets thus constructed achieve the optimal rate of convergence.  Furthermore, we show that it is possible to get some tractability results under certain conditions on the weights and that the fast CBC construction using the fast Fourier transform, as introduced in \cite{NC06a,NC06b}, is also available in our context. This enables us to proceed the CBC construction with $O(dsN \log N)$ operations using $O(N)$ memory. Finally, we conclude this paper with numerical experiments in Section \ref{exp}.

\section{Preliminaries}
\label{pre}
Throughout this paper, we use the following notation. Let $\nat$ be the set of positive integers and let $\nat_0:=\nat \cup \{0\}$. The operators $\oplus$ and $\ominus$ denote the digitwise addition and subtraction modulo $b$, respectively. That is, for $x, x'\in [0,1)$ with $b$-adic expansions $x=\sum_{i=1}^{\infty}\xi_i b^{-i}$ and $x'=\sum_{i=1}^{\infty}\xi'_i b^{-i}$ where $\xi_i,\xi'_i\in \FF_b$, $\oplus$ and $\ominus$ are defined as
  \begin{align*}
    x\oplus x' = \sum_{i=1}^{\infty}\eta_i b^{-i}\ \mbox{and}\ x\ominus x' = \sum_{i=1}^{\infty}\eta'_i b^{-i},
  \end{align*}
where $\eta_i=\xi_i+\xi'_i \pmod b$ and $\eta'_i=\xi_i-\xi'_i \pmod b$, respectively. Similarly, we define digitwise addition and subtraction for non-negative integers based on those $b$-adic expansions. In case of vectors in $[0,1)^s$ or $\nat_0^s$, the operators $\oplus$ and $\ominus$ are applied componentwise. Further we shall use the notation $I_s:=\{1,\ldots,s\}$ for $s\in \nat$ for simplicity.

\subsection{Walsh functions}
Walsh functions were first introduced in \cite{W23} for the case $b=2$ and were generalized later, see for example \cite{C55}. We refer to \cite[Appendix~A]{DP10} for general information on Walsh functions. We first give the definition for the one-dimensional case.
\begin{definition}
Let $b\ge 2$ be an integer and $\omega_b=\exp(2\pi \sqrt{-1}/b)$. Let us denote the $b$-adic expansion of $k\in \nat_0$ by $k = \kappa_0+\kappa_1 b+\cdots +\kappa_{a-1}b^{a-1}$ with $\kappa_0,\ldots,\kappa_{a-1}\in\FF_b$. Then the $k$-th $b$-adic Walsh function ${}_b\wal_k: [0,1)\to \{1,\omega_b,\ldots, \omega_b^{b-1}\}$ is defined as
  \begin{align*}
    {}_b\wal_k(x)
    =
    \omega_b^{\xi_1\kappa_0+\cdots+\xi_a\kappa_{a-1}} ,
  \end{align*}
for $x\in [0,1)$ with its $b$-adic expansion $x=\xi_1 b^{-1} + \xi_2 b^{-2} + \cdots$, that is unique in the sense that infinitely many of the $\xi_i$ are different from $b-1$.
\end{definition}
This definition can be generalized to the multi-dimensional case.

\begin{definition}
For $s\ge 2$, let $\bsx=(x_1,\ldots, x_s)\in [0,1)^s$ and $\bsk=(k_1,\ldots, k_s)\in \nat_0^s$. Then the $\bsk$-th $b$-adic Walsh function ${}_b\wal_{\bsk}: [0,1)^s \to \{1,\omega_b,\ldots, \omega_b^{b-1}\}$ is defined as
  \begin{align*}
    {}_b\wal_{\bsk}(\bsx)
    =
    \prod_{j=1}^s {}_b\wal_{k_j}(x_j) .
  \end{align*}
\end{definition}
Since we shall always use Walsh functions in a fixed base $b$, we omit the subscript and simply write $\wal_k$ or $\wal_{\bsk}$ throughout this paper. We note that the system $\{\wal_{\bsk}: \bsk\in \nat_0^s\}$ is a complete orthonormal system in $\Lcal_2([0,1)^s)$.

\subsection{Polynomial lattice rules}
In this subsection, we introduce the definition of polynomial lattice rules by following the exposition in \cite[Section~10]{DP10}. For a prime $b$, we denote by $\FF_b((x^{-1}))$ the field of formal Laurent series over $\FF_b$. Every element of $\FF_b((x^{-1}))$ has the form
  \begin{align*}
    L = \sum_{l=w}^{\infty}t_l x^{-l} ,
  \end{align*}
for some integer $w$ and $t_l\in \FF_b$. Further, we denote by $\FF_b[x]$ the set of all polynomials over $\FF_b$. For $m\in \nat$, we define the mapping $v_m$ from $\FF_b((x^{-1}))$ to the interval $[0,1)$ by
  \begin{align*}
    v_m\left( \sum_{l=w}^{\infty}t_l x^{-l}\right) =\sum_{l=\max(1,w)}^{m}t_l b^{-l}.
  \end{align*}

We often identify a non-negative integer $k$ whose $b$-adic expansion is given by $k=\kappa_0+\kappa_1 b+\cdots +\kappa_{a-1} b^{a-1}$ with the polynomial over $\FF_b[x]$ as $k(x)=\kappa_0+\kappa_1 x+ \cdots +\kappa_{a-1} x^{a-1}$.  For $\bsk=(k_1,\ldots, k_s)\in (\FF_b[x])^s$ and $\bsq=(q_1,\ldots, q_s)\in (\FF_b[x])^s$, we define the inner product as
  \begin{align*}
     \bsk \cdot \bsq =\sum_{j=1}^{s}k_j q_j \in \FF_b[x] ,
  \end{align*}
and we write $q\equiv 0 \pmod p$ if $p$ divides $q$ in $\FF_b[x]$. Using this notation, a polynomial lattice rule is defined as follows.

\begin{definition}\label{def:polynomial_lattice}
For $m, s \in \nat$, let $p \in \FF_b[x]$ with $\deg(p)=m$ and let $\bsq=(q_1,\ldots,q_s) \in (\FF_b[x])^s$. A polynomial lattice point set is a point set consisting of $b^m$ points given by
  \begin{align*}
    \bsx_n &:= \left( v_m\left( \frac{n(x)q_1(x)}{p(x)} \right) , \ldots , v_m\left( \frac{n(x)q_s(x)}{p(x)} \right) \right) \in [0,1)^s ,
  \end{align*}
for $0\le n<b^m$. A QMC rule using this point set is called a {\em polynomial lattice rule} with generating vector $\bsq$ and modulus $p$.
\end{definition}
Here we note that $p$ and $\bsq$ are not explicitly given and should be chosen properly by users of this rule. In the remainder of this paper, we denote by $P_{b^m,s}(\bsq,p)$ a polynomial lattice point set, implicitly meaning that $\deg(p)=m$ and the number of components in the vector $\bsq$ is $s$. In the subsequent analysis, the concept of the so-called dual net shall play an important role.

\begin{definition}\label{def:dual_net}
Let $m\in \nat$. For $k\in \nat_0$ with its $b$-adic expansion $k=\kappa_0+\kappa_1 b+\cdots+\kappa_{a-1}b^{a-1}$, we denote the truncation of the associated polynomial $k(x)$ by
  \begin{align*}
     \rtr_m(k)(x)=\kappa_0+\kappa_1 x+\cdots+\kappa_{\min(a,m)-1}x^{\min(a,m)-1} .
  \end{align*}
For a polynomial lattice point set $P_{b^m,s}(\bsq,p)$, the dual net is defined as
  \begin{align*}
     D^{\perp}(\bsq,p) := \{ \bsk=(k_1,\ldots,k_s)\in \nat_0^{s}:\ \rtr_m(\bsk) \cdot \bsq\equiv 0 \pmod p \} ,
  \end{align*}
where $\rtr_m(\bsk)=(\rtr_m(k_1),\ldots, \rtr_m(k_s))$.
\end{definition}

Furthermore, we shall use the following two lemmas in this paper. The first lemma bridges between a polynomial lattice point set $P_{b^m,s}(\bsq,p)$ and Walsh functions. The proof is straightforward from the above definition of $D^{\perp}(\bsq,p)$, \cite[Lemma~10.6]{DP10} and \cite[Lemma~4.75]{DP10}. The second lemma implies that any polynomial lattice point set $P_{b^m,s}(\bsq,p)$ is a digital net. The proof is also straightforward from \cite[Lemma~4.72]{DP10}.

\begin{lemma}\label{lemma:dual_walsh}
Let $P_{b^m,s}(\bsq,p)$ be a polynomial lattice point set and $D^{\perp}(\bsq,p)$ be its dual net. Then we have
  \begin{align*}
     \frac{1}{b^m}\sum_{n=0}^{b^m-1}\wal_{\bsk}(\bsx_n)=\left\{ \begin{array}{ll}
     1 & \mbox{if}\ \bsk\in D^{\perp}(\bsq,p) , \\
     0 & \mbox{otherwise} . \\
     \end{array} \right.
  \end{align*}
\end{lemma}

\begin{lemma}\label{lemma:group}
For a prime $b$, any polynomial lattice point set $P_{b^m,s}(\bsq,p)$ is homomorphic to an $\FF_b$-linear subspace of $\FF_b^{s\times \infty}$ with possible multiplicity.
\end{lemma}

\subsection{Higher order digital nets}
As mentioned in the previous section, QMC rules using higher order digital nets can exploit the smoothness of an integrand so that they achieve the optimal convergence rate of the integration error for $\alpha$-smooth functions, where $\alpha \ge 2$ is an integer. This result is based on a bound on the decay of the Walsh coefficients of smooth functions \cite{Dic08}. An explicit construction of higher order digital nets by means of a digit interlacing function was given in \cite{Dic07,Dic08}. Although we have already mentioned this function in the previous section, we describe the interlacing algorithm in more detail in the following.

Since digit interlacing is applied to each point separately, we use just one generic point to describe the procedure. For $s,d\in \nat$, let $\bsy= (y_{1},\ldots, y_{ds}) \in [0,1)^{ds}$ and let us denote the $b$-adic expansion of each coordinate $y_j$ by
  \begin{align*}
    y_j= \frac{\eta_{j,1}}{b} + \frac{\eta_{j,2}}{b^2} + \cdots,
  \end{align*}
which is unique in the sense that infinitely many digits are different from $b-1$. We then obtain a point $\bsx \in [0,1)^{s}$ by digitally interlacing non-overlapping $d$ components of $\bsy$ in the following way. Let $\bsx =(x_{1},\ldots, x_{s})$ where
  \begin{align*}
    x_j = \sum_{a=1}^{\infty}\sum_{r=1}^{d}\frac{\eta_{d(j-1)+r,a}}{b^{r+(a-1)d}},
  \end{align*}
for $1\le j\le s$. That is, for $a,r\in \nat$ with $1\le r\le d$, the $(r+(a-1)d)$-th digit in the $b$-adic expansion of $x_j$ equals the $a$-th digit in the $b$-adic expansion of $y_{d(j-1)+r}$. We denote the above mapping by $\Dcal_{d}: [0,1)^d \to [0,1)$ and we simply write $x_{j}=\Dcal_{d}(y_{d(j-1)+1},\ldots, y_{dj})$. Further we write $$\bsx=\Dcal_d(\bsy) := (\Dcal_d(y_{1},\ldots, y_d), \ldots, \Dcal_d(y_{d(s-1)+1},\ldots, y_{ds})),$$ when $\bsx$ is obtained by interlacing the components of $\bsy$. Thus it is obvious that in order to construct a good higher order digital net consisting of $N$ points in $[0,1)^{s}$, we need to choose suitable $N$ points in $[0,1)^{ds}$.

In this paper, we are interested in using polynomial lattice point sets as point sets in $[0,1)^{ds}$ that are used for interlaced components to construct higher order digital nets. For clarity, we give the definition of higher order digital nets based on polynomial lattice point sets.

\begin{definition}\label{def:interlacing_polynomial_lattice}
Let $b$ be a prime and let $m, s, d \in \nat$. Let $p \in \FF_b[x]$ with $\deg(p)=m$ and let $\bsq=(q_1,\ldots,q_{ds}) \in (\FF_b[x])^{ds}$. A higher order digital net consisting of $b^m$ points $\{\bsx_0, \ldots, \bsx_{b^m-1}\} \subset [0,1)^s$ is constructed as follows. For $0\le n<b^m$, the $n$-th point $\bsx_n$ is obtained by
  \begin{align*}
    \bsx_n = \Dcal_d(\bsy_n) ,
  \end{align*}
where the point $\bsy_n$ is the $n$-th point of a polynomial lattice point set in dimension $ds$, $P_{b^m,ds}(\bsq,p)$, which is given as
  \begin{align*}
    \bsy_n = \left( v_m\left( \frac{n(x) \, q_1(x)}{p(x)} \right) , \ldots , v_m\left( \frac{n(x) \, q_{ds}(x)}{p(x)} \right) \right) \in [0,1)^{ds}.
  \end{align*}
\end{definition}

In this construction algorithm, the search for good $b^m$ points in $[0,1)^{ds}$ has now been reduced to finding good polynomials $p$ and $\bsq=(q_1,\ldots,q_{ds})$.

Randomization of QMC point sets is useful to obtain some statistical information on the integration error. Especially for randomization of higher order digital nets, two algorithms have been discussed in the literature. One is a random digital shift, see \cite{BD09,Dic07}, and the other is a higher order scrambling that is a generalization of Owen's scrambling introduced in \cite{Owe95}, see \cite{Dic11,GDxx}. Since we are concerned with the former in this paper, we follow the exposition in \cite[Section~6]{Dic07} to introduce some basic concepts of a random digital shift.

Let $P_{N}=\{\bsx_0,\ldots, \bsx_{N-1}\}$ be a point set in  $[0,1)^s$ with $\bsx_n=(x_{n,1},\ldots, x_{n,s})$. Let $\bssigma=(\sigma_1,\ldots,\sigma_s)$ be a point in $[0,1)^s$ such that each $\sigma_j$ is independently and uniformly distributed in $[0,1)$. Then a digitally shifted point set $P_{N, \bssigma}=\{\bsz_0,\ldots, \bsz_{N-1}\}$ is given by
  \begin{align*}
    \bsz_n = \bsx_n\oplus \bssigma ,
  \end{align*}
for $0\le n<N$. In order to calculate the mean square worst-case error in the next section, the following lemma shall be required. We refer to \cite[Lemma~16.37]{DP10} for the proof.

\begin{lemma}\label{lemma:digital_shift}
For any two points $x_1,x_2\in [0,1)$, we have
  \begin{align*}
     \int_{0}^{1}\wal_k(x_1\oplus \sigma)\overline{\wal_l(x_2\oplus \sigma)}\rd \sigma= \left\{ \begin{array}{ll}
     \wal_k(x_1\ominus x_2) & \mbox{if}\ k=l , \\
     0 & \mbox{otherwise} . \\
     \end{array} \right.
  \end{align*}
\end{lemma}

\subsection{Weighted Sobolev spaces of high order}

Here we follow the exposition in \cite{BD09} and introduce weighted reproducing kernel Hilbert spaces that are considered in this paper. The concept of weighted spaces was introduced in \cite{SW98}, where the {\em weights} play a role in moderating the importance of different variables or groups of variables in function spaces. From the viewpoint of information-based complexity, it provides an essential insight into tractability properties of multivariate problems. We refer to \cite{NW08,NW10} for general information on tractability of multivariate problems.

Let us start with the one-dimensional unweighted case. The inner product of the Hilbert space $H_{1,\alpha,(1)}$ is defined for $\alpha\ge 2$, $\alpha\in \nat$, by
  \begin{align*}
    \langle f,g\rangle_{H_{1,\alpha,(1)}}:=\sum_{r=0}^{\alpha-1}\int_0^1 f^{(r)}(x)\rd x\int_0^1 g^{(r)}(x)\rd x+\int_0^1 f^{(\alpha)}(x)g^{(\alpha)}(x)\rd x ,
  \end{align*}
where we denote by $f^{(r)}$ the $r$-th derivative of $f$ and set $f^{(0)}=f$. Let $\| f\|_{H_{1,\alpha,(1)}}:=\sqrt{\langle f,f\rangle_{H_{1,\alpha,(1)}}}$ be the norm of $f$ associated with $H_{1,\alpha,(1)}$. We note that all functions in $H_{1,\alpha,(1)}$ are defined on $[0,1)$.

We now define the function $K_{1,\alpha,(1)}:[0,1)\times [0,1)\to \RR$ by
  \begin{align*}
    K_{1,\alpha,(1)}(x,y):=\sum_{r=1}^{\alpha}\frac{B_{r}(x)B_{r}(y)}{(r !)^2}+(-1)^{\alpha+1}\frac{B_{2\alpha}(|x-y|)}{(2\alpha)!} ,
  \end{align*}
where $B_{r}$ denotes the Bernoulli polynomial of degree $r$. We refer to \cite[Chapter~23]{AS71} for information on Bernoulli polynomials. As shown in \cite[Section~2.2]{BD09}, $K_{1,\alpha,(1)}$ has the following property:
  \begin{align}\label{eq:kernel_prop}
    \int_{0}^{1}K_{1,\alpha,(1)}(x,y)\rd x=0 ,
  \end{align}
for any $0\le y<1$. The reproducing kernel for the space $H_{1,\alpha,(1)}$ is given by $1+K_{1,\alpha,(1)}$. That is, for any $f\in H_{1,\alpha,(1)}$, we have 
  \begin{align*}
    f(x) =\langle f,1+K_{1,\alpha,(1)}(\cdot,x)\rangle_{H_{1,\alpha,(1)}} ,
  \end{align*}
for $0\le x<1$. From the definition of the inner product, it is clear that the space $H_{1,\alpha,(1)}$ consists of univariate functions whose derivatives up to order $\alpha\ge 2$, $\alpha \in \nat$, are square integrable.

We consider the multi-dimensional weighted case next. Let $I_s:=\{1,\ldots,s\}$ be the index set for $s\in \nat$. The inner product of the $s$-dimensional weighted unanchored Sobolev space $H_{s,\alpha,\bsgamma}$ of order $\alpha\ge 2$, $\alpha\in \nat$, for a set of non-negative weights $\bsgamma=(\gamma_u)_{u\subseteq I_s}$ is defined by
  \begin{align*}
    \langle f,g\rangle_{H_{s,\alpha,\bsgamma}}
    := & \sum_{u\subseteq I_s}\gamma_u^{-1}\sum_{v\subseteq u}\sum_{\bsr_{u\setminus v}\in \{1,\ldots, \alpha-1\}^{|u\setminus v|}} \\
    & \times \int_{[0,1)^{|v|}}\int_{[0,1)^{s-|v|}}f^{(\bsr_{u\setminus v},\bsalpha_v,\bszero)}(\bsx)\rd\bsx_{-v}\int_{[0,1)^{s-|v|}}g^{(\bsr_{u\setminus v},\bsalpha_v,\bszero)}(\bsx)\rd\bsx_{-v}\rd\bsx_{v} ,
  \end{align*}
where we use the following notation. For $\bsr_{u\setminus v}=(r_j)_{j\in u\setminus v}$, we denote by $(\bsr_{u\setminus v},\bsalpha_v,\bszero)$ the $s$-dimensional vector in which the $j$-th component is $r_j$ for $j\in u\setminus v$, $\alpha$ for $j\in v$, and 0 for $I_s\setminus u$. For $u\subseteq I_s$ such that $\gamma_u=0$, we assume that the corresponding inner double sum equals 0 and we set $0/0=0$. Let $\| f\|_{H_{s,\alpha,\bsgamma}}=\sqrt{\langle f,f\rangle_{H_{s,\alpha,\bsgamma}}}$ be again the norm of $f$ associated with $H_{s,\alpha,\bsgamma}$. 

The reproducing kernel for $H_{s,\alpha,\bsgamma}$ becomes
  \begin{align*}
    K_{s,\alpha,\bsgamma}(\bsx,\bsy) & = \sum_{u\subseteq I_s}\gamma_u \prod_{j\in u}K_{1,\alpha,(1)}(x_j,y_j) \\
    & := \sum_{u\subseteq I_s}\gamma_u \prod_{j\in u}\left( \sum_{r=1}^{\alpha}\frac{B_{r}(x_j)B_{r}(y_j)}{(r !)^2}+(-1)^{\alpha+1}\frac{B_{2\alpha}(|x_j-y_j|)}{(2\alpha)!}\right) ,
  \end{align*}
where we set 
  \begin{align*}
    \prod_{j\in \emptyset}K_{1,\alpha,(1)}(x_j,y_j) = 1 .
  \end{align*}
That is, for any $f\in H_{s,\alpha,\bsgamma}$, we have
  \begin{align*}
    f(\bsx) =\langle f,K_{s,\alpha,\bsgamma}(\cdot,\bsx)\rangle_{H_{s,\alpha,\bsgamma}} ,
  \end{align*}
for $\bsx\in [0,1)^s$. As in the one-dimensional unweighted case, the space $H_{s,\alpha,\bsgamma}$ consists of $\alpha$-smooth functions, that is, multivariate functions whose mixed partial derivatives up to order $\alpha\ge 2$, $\alpha \in \nat$, in each variable are square integrable.

\section{Mean square worst-case error}
\label{bound}

In this section, we derive an upper bound on the mean square worst-case error in the space $H_{s,\alpha,\bsgamma}$ for digitally shifted higher order digital nets. To begin with, the worst-case error of a QMC rule using a point set $P_N=\{\bsx_0,\ldots,\bsx_{N-1}\}$ in the space $H_{s,\alpha,\bsgamma}$ is defined by
  \begin{align*}
    e(P_N,H_{s,\alpha,\bsgamma}) := \sup_{\substack{f\in H_{s,\alpha,\bsgamma} \\ \| f \|_{H_{s,\alpha,\bsgamma}} \le 1}} |I(f)-Q(f;P_N)| .
  \end{align*}
The initial error is given as
  \begin{align*}
    e(P_0,H_{s,\alpha,\bsgamma}) := \sup_{\substack{f\in H_{s,\alpha,\bsgamma} \\ \| f \|_{H_{s,\alpha,\bsgamma}} \le 1}} |I(f)| .
  \end{align*}
From \cite[Theorem~13]{BD09}, we have simple expressions for the squared worst-case error and the squared initial error in the space $H_{s,\alpha,\bsgamma}$, namely
  \begin{align*}
    e^2(P_N,H_{s,\alpha,\bsgamma}) = & \int_{[0,1)^s}\int_{[0,1)^s}K_{s,\alpha,\bsgamma}(\bsx,\bsy)\rd\bsx\rd\bsy \\
    & - \frac{2}{N}\sum_{n=0}^{N-1}\int_{[0,1)^s}K_{s,\alpha,\bsgamma}(\bsx_n,\bsx)\rd\bsx+\frac{1}{N^2}\sum_{n,n'=0}^{N-1}K_{s,\alpha,\bsgamma}(\bsx_n,\bsx_{n'}) \\
    = & -\gamma_{\emptyset}+\frac{1}{N^2}\sum_{n,n'=0}^{N-1}K_{s,\alpha,\bsgamma}(\bsx_n,\bsx_{n'}) ,
  \end{align*}
and
  \begin{align*}
    e^2(P_0,H_{s,\alpha,\bsgamma}) = \int_{[0,1)^s}\int_{[0,1)^s}K_{s,\alpha,\bsgamma}(\bsx,\bsy)\rd\bsx\rd\bsy = \gamma_{\emptyset} ,
  \end{align*}
respectively. Let us consider a randomization of $P_N$ by using a random digital shift. The mean square worst-case error for a digitally shifted point set $P_{N,\bssigma}$, denoted by $\tilde{e}^2(P_N,H_{s,\alpha,\bsgamma})$, is defined and calculated as
  \begin{align}\label{eq:mse0}
    \tilde{e}^2(P_N,H_{s,\alpha,\bsgamma}) := & \int_{[0,1)^s}e^2(P_{N, \bssigma} ,H_{s,\alpha,\bsgamma})\rd \bssigma \nonumber \\
    = & -\gamma_{\emptyset}+\frac{1}{N^2}\sum_{n,n'=0}^{N-1}\int_{[0,1)^s}K_{s,\alpha,\bsgamma}(\bsx_n\oplus \bssigma, \bsx_{n'}\oplus \bssigma)\rd \bssigma \nonumber \\
    = & \frac{1}{N^2}\sum_{n,n'=0}^{N-1}\sum_{\emptyset \ne u\subseteq I_s}\gamma_u \prod_{j\in u}\int_{0}^{1}K_{1,\alpha,(1)}(x_{n,j}\oplus \sigma_j, x_{n',j}\oplus \sigma_j)\rd \sigma_j.
  \end{align}
On the other hand, the mean square initial error is given as
  \begin{align*}
    \tilde{e}^2(P_0,H_{s,\alpha,\bsgamma}) = \gamma_{\emptyset} .
  \end{align*}

We now consider the Walsh series expansion of $K_{1,\alpha,(1)}(x, y)$ as follows.
  \begin{align*}
    K_{1,\alpha,(1)}(x, y) = \sum_{k,l=0}^{\infty}\hat{K}_{1,\alpha,(1)}(k,l)\wal_k(x)\overline{\wal_l(y)} ,
  \end{align*}
for any $x,y\in [0,1)$, where $\hat{K}_{1,\alpha,(1)}(k,l)$ denotes the $(k,l)$-th Walsh coefficient which is defined by
  \begin{align*}
    \hat{K}_{1,\alpha,(1)}(k,l) := \int_0^1 \int_0^1 K_{1,\alpha,(1)}(x, y)\overline{\wal_k(x)}\wal_l(y)\rd x\rd y .
  \end{align*}
We refer to \cite[Appendix~A.3]{DP10} for a discussion on the pointwise absolute convergence of Walsh series. By using the property (\ref{eq:kernel_prop}), we have 
  \begin{align*}
    \hat{K}_{1,\alpha,(1)}(k,0) = \hat{K}_{1,\alpha,(1)}(0,l) =0 ,
  \end{align*}
for any $k,l\in \nat_0$. Thus, the Walsh series expansion of $K_{1,\alpha,(1)}(x, y)$ reduces to
  \begin{align*}
    K_{1,\alpha,(1)}(x, y) = \sum_{k,l=1}^{\infty}\hat{K}_{1,\alpha,(1)}(k,l)\wal_k(x)\overline{\wal_l(y)} .
  \end{align*}

Substituting the above Walsh series expansion of $K_{1,\alpha,(1)}(x, y)$ into (\ref{eq:mse0}) and then using Lemma \ref{lemma:digital_shift}, we have
  \begin{align}\label{eq:mse}
    & \tilde{e}^2(P_N,H_{s,\alpha,\bsgamma}) \nonumber \\
    = & \frac{1}{N^2}\sum_{n,n'=0}^{N-1}\sum_{\emptyset \ne u\subseteq I_s}\gamma_u \prod_{j\in u}\int_{0}^{1}\sum_{k_j,l_j=1}^{\infty}\hat{K}_{1,\alpha,(1)}(k_j,l_j)\wal_{k_j}(x_{n,j}\oplus \sigma_j)\overline{\wal_{l_j}(x_{n',j}\oplus \sigma_j)}\rd \sigma_j \nonumber \\
    = & \frac{1}{N^2}\sum_{n,n'=0}^{N-1}\sum_{\emptyset \ne u\subseteq I_s}\gamma_u \prod_{j\in u}\sum_{k_j,l_j=1}^{\infty}\hat{K}_{1,\alpha,(1)}(k_j,l_j)\int_{0}^{1}\wal_{k_j}(x_{n,j}\oplus \sigma_j)\overline{\wal_{l_j}(x_{n',j}\oplus \sigma_j)}\rd \sigma_j \nonumber \\
    = & \frac{1}{N^2}\sum_{n,n'=0}^{N-1}\sum_{\emptyset \ne u\subseteq I_s}\gamma_u \prod_{j\in u}\sum_{k_j=1}^{\infty}\hat{K}_{1,\alpha,(1)}(k_j)\wal_{k_j}(x_{n,j}\ominus x_{n',j}) \nonumber \\
    = & \sum_{\emptyset \ne u\subseteq I_s}\gamma_u \sum_{\bsk_u\in \nat^{|u|}}\hat{K}_{u,\alpha,(1)}(\bsk_u)\frac{1}{N^2}\sum_{n,n'=0}^{N-1}\wal_{(\bsk_u,\bszero)}(\bsx_{n}\ominus \bsx_{n'}) ,
  \end{align}
where we write $\hat{K}_{1,\alpha,(1)}(k)=\hat{K}_{1,\alpha,(1)}(k,k)$ and $\hat{K}_{u,\alpha,(1)}(\bsk_u)=\prod_{j\in u}\hat{K}_{1,\alpha,(1)}(k_j)$ for short, and we denote by $(\bsk_u,\bszero)$ the $s$-dimensional vector in which the $j$-th component is $k_j$ if $j\in u$, and $0$ otherwise.

\subsection{A bound on Walsh coefficients for smooth functions}
From \cite[(13) and Proposition~20]{BD09} in which the case $k=l$ is considered here, we have the following bound on the Walsh coefficient $\hat{K}_{1,\alpha,(1)}(k)$
  \begin{align*}
    \left| \hat{K}_{1,\alpha,(1)}(k)\right| \le D_{\alpha,b}b^{-2\mu_{\alpha}(k)} ,
  \end{align*}
for $k\in \nat$, where $\mu_{\alpha}(k)$ is the weight introduced in \cite{Dic08}, which is defined as
  \begin{align}\label{eq:dick_weight}
    \mu_{\alpha}(k) := a_1+\cdots +a_{\min(\nu, \alpha)} ,
  \end{align}
where the $b$-adic expansion of $k$ is given as $k=\kappa_1 b^{a_1-1}+\cdots +\kappa_{\nu}b^{a_\nu-1}$ such that $0<\kappa_1,\ldots,\kappa_{\nu}<b$ and $a_1>\cdots > a_{\nu}>0$. We define $\mu_{\alpha}(0):=0$. Moreover, $D_{\alpha,b}$ is positive and depends only on $\alpha$ and $b$, given explicitly as
  \begin{align*}
    D_{\alpha,b} = \max_{1\le \nu \le \alpha}\left( C'_{\alpha,b,\nu}+\tilde{C}_{2\alpha,b}b^{-2(\alpha-\nu)}\right) ,
  \end{align*}
in which $C'_{\alpha,b,\nu}$ and $\tilde{C}_{2\alpha,b}$ are respectively given by
  \begin{align*}
    C'_{\alpha,b,\nu} = \sum_{\tau=\nu}^{\alpha}C_{\tau,b}^2 b^{-2(\tau-\nu)} ,
  \end{align*}
where $C_{1,b}=(2\sin(\pi/b))^{-1}$ and $C_{\tau,b}=(1+1/b+1/(b(b+1)))^{\tau-2}/(2\sin(\pi/b))^{\tau}$ for $\tau\ge 2$, and 
  \begin{align*}
    \tilde{C}_{2\alpha,b} = \frac{2(1+1/b+1/(b(b+1)))^{2\alpha-2}}{(2\sin(\pi/b))^{2\alpha}} .
  \end{align*}
Furthermore, we write $\mu_{\alpha}(\bsk_u)=\sum_{j\in u}\mu_{\alpha}(k_j)$. Then we have
  \begin{align*}
    \left| \hat{K}_{u,\alpha,(1)}(\bsk_u)\right| = \prod_{j\in u}\left| \hat{K}_{1,\alpha,(1)}(k_j)\right| \le D_{\alpha,b}^{|u|}b^{-2\mu_{\alpha}(\bsk_u)} ,
  \end{align*}
where $|u|$ denotes the cardinality of $u$.

Applying the triangle inequality to (\ref{eq:mse}) and substituting the above bound on the Walsh coefficient into it, we have
  \begin{align}\label{eq:mse2}
    \tilde{e}^2(P_N,H_{s,\alpha,\bsgamma}) \le & \sum_{\emptyset \ne u\subseteq I_s}\gamma_u \sum_{\bsk_u\in \nat^{|u|}}\left|\hat{K}_{1,\alpha,(1)}(\bsk_u)\frac{1}{N^2}\sum_{n,n'=0}^{N-1}\wal_{(\bsk_u,\bszero)}(\bsx_{n}\ominus \bsx_{n'})\right| \nonumber \\
    \le & \sum_{\emptyset \ne u\subseteq I_s}\gamma_u D_{\alpha,b}^{|u|}\sum_{\bsk_u\in \nat^{|u|}}b^{-2\mu_{\alpha}(\bsk_u)}\left| \frac{1}{N^2}\sum_{n,n'=0}^{N-1}\wal_{(\bsk_u, \bszero)}(\bsx_n\ominus \bsx_{n'})\right| .
  \end{align}
We note that this bound on $\tilde{e}^2(P_N,H_{s,\alpha,\bsgamma})$ holds for any point set $P_N$.

\subsection{A bound on the mean square worst-case error}
Here we give an upper bound on $\tilde{e}^2(P_N,H_{s,\alpha,\bsgamma})$ where we restrict the class of point sets to higher order digital nets constructed according to Definition \ref{def:interlacing_polynomial_lattice}, that is, we suppose that a point set $P_N$ in (\ref{eq:mse2}) is given by digitally interlacing a polynomial lattice point set in dimension $ds$.

\begin{theorem}\label{thm:upper_bound}
Let $b$ be a prime. Let $m,s,\alpha,d\in \nat$ such that $\alpha\ge 2$, and let $\bsgamma=(\gamma_u)_{u\subseteq I_s}$ be a set of weights. Let $P_{b^m,ds}(\bsq,p)$ be a polynomial lattice point set, and let $D^{\perp}(\bsq,p)$ be its dual net. For a higher order digital net $P_{b^m}$ constructed according to Definition \ref{def:interlacing_polynomial_lattice}, the mean square worst-case error in $H_{s,\alpha,\bsgamma}$ with respect to a random digital shift is bounded as
  \begin{align*}
    \tilde{e}^2(P_{b^m},H_{s,\alpha,\bsgamma}) \le \sum_{\emptyset \ne w\subseteq I_{ds}}\gamma_{\phi(w)} \tilde{D}_{\alpha,b,d}^{|\phi(w)|}\sum_{\substack{\bsl_w\in \nat^{|w|} \\ (\bsl_w,\bszero)\in D^{\perp}(\bsq,p)}}r_{\alpha,d}(\bsl_w) ,
  \end{align*}
where $\tilde{D}_{\alpha,b,d}=b^{(2d-1)\alpha}D_{\alpha,b}$ is positive and depends only on $\alpha,b$ and $d$, and for $\emptyset \ne w\subseteq I_{ds}$, $(\bsl_w,\bszero)$ denotes the $ds$-dimensional vector whose $j$-th component is $l_j$ if $j\in w$, and $0$ otherwise, and define the mapping $\phi: I_{ds}\to I_s$ as 
  \begin{align*}
    \phi(w):=\{1\le j\le s: w\cap\{d(j-1)+1,\ldots, dj\}\ne \emptyset \},
  \end{align*}
and further $r_{\alpha,d}(\bsl_w):=\prod_{j\in w}r_{\alpha,d}(l_j)$, where
  \begin{align*}
    r_{\alpha,d}(l_j) := \left\{ \begin{array}{ll}
     1 & \mbox{if}\ l_j=0 , \\
     b^{-2\min(\alpha,d)\mu_1(l_j)-\alpha} & \mbox{otherwise} . \\
     \end{array} \right.
  \end{align*}
\end{theorem}

To emphasize the role of the polynomial lattice point set, we denote the upper bound shown in the above theorem by 
  \begin{align*}
    B_{\alpha,d,\bsgamma}(\bsq,p) = \sum_{\emptyset \ne w\subseteq I_{ds}}\gamma_{\phi(w)} \tilde{D}_{\alpha,b,d}^{|\phi(w)|}\sum_{\substack{\bsl_w\in \nat^{|w|} \\ (\bsl_w,\bszero)\in D^{\perp}(\bsq,p)}}r_{\alpha,d}(\bsl_w) ,
  \end{align*}
where $p\in \FF_b[x]$ such that $\deg(p)=m$ and $\bsq\in (\FF_b[x])^{ds}$. The following corollary gives a computable formula for $B_{\alpha,d,\bsgamma}(\bsq,p)$. We write $\log_b$ for the logarithm in base $b$ and set $b^{(2\min(\alpha,d)-1)\lfloor \log_b 0\rfloor}=0$.

\begin{corollary}\label{cor:upper_bound}
Let $b$ be a prime. Let $m,s,\alpha,d\in \nat$ such that $\alpha\ge 2$, and let $\bsgamma=(\gamma_u)_{u\subseteq I_s}$ be a set of weights. Let $B_{\alpha,d,\bsgamma}(\bsq,p)$ be given as above, and let $P_{b^m,ds}(\bsq,p)=\{\bsy_0,\ldots,\bsy_{b^m-1}\}\subset [0,1)^{ds}$ be a polynomial lattice point set. 
\begin{enumerate}
\item We have
  \begin{align*}
    B_{\alpha,d,\bsgamma}(\bsq,p) = \frac{1}{b^m}\sum_{n=0}^{b^m-1}\sum_{\emptyset \ne u\subseteq I_{s}}\gamma_u \tilde{D}_{\alpha,b,d}^{|u|} \prod_{j\in u}\left[ -1+\prod_{l=1}^{d}\left( 1+\chi_{\alpha,d}(y_{n,d(j-1)+l})\right) \right] ,
  \end{align*}
where we define for $y\in [0,1)$
  \begin{align*}
    \chi_{\alpha,d}(y)=\frac{b-1-b^{(2\min(\alpha,d)-1)\lfloor \log_b y\rfloor}(b^{2\min(\alpha,d)}-1)}{b^{\alpha}(b^{2\min(\alpha,d)}-b)} .
  \end{align*}
\item Particularly in case of product weights, that is $\gamma_u=\prod_{j\in u}\gamma_j$ for all $u\subseteq I_s$, we have
  \begin{align*}
    B_{\alpha,d,\bsgamma}(\bsq,p) = -1+\frac{1}{b^m}\sum_{n=0}^{b^m-1}\prod_{j=1}^{s}\left[1-\gamma_j \tilde{D}_{\alpha,b,d} +\gamma_j \tilde{D}_{\alpha,b,d}\prod_{l=1}^{d}\left( 1+\chi_{\alpha,d}(y_{n,d(j-1)+l})\right) \right] .
  \end{align*}
\end{enumerate}
\end{corollary}
The proofs of Theorem \ref{thm:upper_bound} and Corollary \ref{cor:upper_bound} are given in Appendix \ref{app:a}.

\section{Component-by-component construction of polynomial lattice point sets}
\label{cbc}

\subsection{Construction algorithm}
As an efficient computer search algorithm to find good polynomials $p$ and $\bsq=(q_1,\ldots,q_{ds})$ such that $B_{\alpha,d,\bsgamma}(\bsq,p)$ becomes small, we investigate the CBC construction. We denote by $R_{b,m}$ the set of all non-zero polynomials over $\FF_b$ with degree less than $m$, that is,
	\begin{align*}
		R_{b,m}=\{ q\in \FF_b[x]: \deg(q)<m\ \mbox{and}\ q\ne0\} .
	\end{align*}
We search $\bsq$ from $R_{b,m}^{ds}$ component by component. If $p$ is irreducible, every one-dimensional projection of the point set $P_{b^m,ds}(\bsq,p)$ consists of the equidistributed points $0,1/b^m,\ldots,(b^m-1)/b^m$ for any $\bsq\in R_{b,m}^{ds}$. Thus, without loss of generality we can restrict ourselves to considering $q_1=1$. Thus, the CBC construction proceeds as follows.

\begin{algorithm}\label{algorithm:cbc}
Let $b$ be a prime. For $m,s,\alpha,d\in \nat$ with $\alpha\ge 2$ and a set of weights $\bsgamma=(\gamma_u)_{u\subseteq I_s}$, do the following:
	\begin{enumerate}
		\item Choose an irreducible polynomial $p\in \FF_b[x]$ such that $\deg(p)=m$.
		\item Set $q_1=1$.
		\item For $r=2,\ldots, ds$, find $q^*_r=q_r$ by minimizing $B_{\alpha,d, \bsgamma}((\bsq_{r-1}, q^*_r),p)$ as a function of $q^*_r\in R_{b,m}$ where $\bsq_{r-1}=(q_1,\ldots,q_{r-1})$ and 
  \begin{align*}
    B_{\alpha,d, \bsgamma}((\bsq_{r-1}, q^*_r),p) = \sum_{\emptyset \ne w\subseteq I_r}\gamma_{\phi(w)} \tilde{D}_{\alpha,b,d}^{|\phi(w)|}\sum_{\substack{\bsl_w\in \nat^{|w|} \\ (\bsl_w,\bszero)\in D^{\perp}((\bsq_{r-1}, q^*_r),p)}}r_{\alpha,d}(\bsl_w) .
  \end{align*}
	\end{enumerate}
\end{algorithm}

\begin{remark}\label{remark:cbc_compute}
In the third step of Algorithm \ref{algorithm:cbc}, we need to compute $B_{\alpha,d, \bsgamma}((\bsq_{r-1}, q^*_r),p)$ for which we have a computable formula as shown below. Since the proof is almost the same as that of Corollary \ref{cor:upper_bound}, we omit it.

For $1\le r\le ds$, let $P_{b^m,r}((\bsq_{r-1}, q^*_r),p)=\{\bsy_0,\ldots,\bsy_{b^m-1}\}\subset [0,1)^r$ be a polynomial lattice point set, where $\bsq_{0}$ is the empty set. We write $r=d(j_1-1)+d_1$ such that $j_1,d_1\in \nat$ and $d_1\in \{1,\ldots, d\}$. Then we have
  \begin{align*}
    & B_{\alpha,d,\bsgamma}((\bsq_{r-1}, q^*_r),p) \\
    = & \frac{1}{b^m}\sum_{n=0}^{b^m-1}\sum_{\emptyset \ne u\subseteq I_{j_1-1}}\gamma_u \tilde{D}_{\alpha,b,d}^{|u|} \prod_{j\in u}\left[ -1+\prod_{l=1}^{d}\left( 1+\chi_{\alpha,d}(y_{n,d(j-1)+l})\right) \right] \\
    & + \frac{1}{b^m}\sum_{n=0}^{b^m-1}\sum_{u\subseteq I_{j_1-1}}\gamma_{u\cup \{j_1\}} \tilde{D}_{\alpha,b,d}^{|u|+1} \prod_{j\in u}\left[ -1+\prod_{l=1}^{d}\left( 1+\chi_{\alpha,d}(y_{n,d(j-1)+l})\right) \right] \\
    & \times \left[ -1+\prod_{l=1}^{d_1}\left( 1+\chi_{\alpha,d}(y_{n,d(j_1-1)+l})\right) \right] .
  \end{align*}
Particularly in case of product weights, we have
  \begin{align*}
    & B_{\alpha,d,\bsgamma}((\bsq_{r-1}, q^*_r),p) \\
    = & -1+\frac{1}{b^m}\sum_{n=0}^{b^m-1}\prod_{j=1}^{j_1-1}\left[1-\gamma_j \tilde{D}_{\alpha,b,d} +\gamma_j \tilde{D}_{\alpha,b,d}\prod_{l=1}^{d}\left( 1+\chi_{\alpha,d}(y_{n,d(j-1)+l})\right) \right] \\
    & \times \left[1-\gamma_{j_1} \tilde{D}_{\alpha,b,d} +\gamma_{j_1} \tilde{D}_{\alpha,b,d}\prod_{l=1}^{d_1}\left( 1+\chi_{\alpha,d}(y_{n,d(j_1-1)+l})\right) \right] .
  \end{align*}
\end{remark}

The following theorem gives a bound on $B_{\alpha,d, \bsgamma}(\bsq_r,p)$ for $1\le r\le ds$, which justifies the CBC construction.

\begin{theorem}\label{theorem:cbc_bound}
Let $b$ be a prime. For $m,s,\alpha,d\in \nat$ with $\alpha\ge 2$ and a set of weights $\bsgamma=(\gamma_u)_{u\subseteq I_s}$, let $p\in \FF_b[x]$ and $\bsq=(q_1,\ldots,q_{ds})\in (\FF_b[x])^{ds}$ be found by Algorithm \ref{algorithm:cbc}. Then for any $1\le r\le ds$ we have
  \begin{align}\label{eq:cbc_bound}
    B_{\alpha,d,\bsgamma}(\bsq_r,p) \le \frac{1}{(b^m-1)^{1/\lambda}}\left[ \sum_{\emptyset \ne u \subseteq I_{j_1-1}}\gamma_u^{\lambda}G_{\alpha,d,\lambda,d}^{|u|}+G_{\alpha,d,\lambda,d_1}\sum_{u \subseteq I_{j_1-1}}\gamma_{u\cup \{j_1\}}^{\lambda}G_{\alpha,d,\lambda,d}^{|u|} \right]^{1/\lambda} ,
  \end{align}
for $1/(2\min(\alpha,d)) < \lambda \le 1$, where we write $r=d(j_1-1)+d_1$ such that $j_1,d_1\in \nat$ and $d_1\in \{1,\ldots, d\}$, and we denote for $1\le a\le d$
  \begin{align*}
    G_{\alpha,d,\lambda,a} = \tilde{D}_{\alpha,b,d}^{\lambda}\left( -1+(1+C_{\alpha,d,\lambda})^a\right) ,
  \end{align*}
where $\tilde{D}_{\alpha,b,d}$ is given as in Theorem \ref{thm:upper_bound} and 
  \begin{align*}
    C_{\alpha,d,\lambda} = \frac{1}{b^{\alpha \lambda}}\max\left\{ \left( \frac{b-1}{b^{2\min(\alpha,d)}-b}\right)^{\lambda}, \frac{b-1}{b^{2\lambda\min(\alpha,d)}-b}\right\} .
  \end{align*}
\end{theorem}
The proof of Theorem \ref{theorem:cbc_bound} is given in Appendix \ref{app:b}.

\begin{remark}\label{remark:cbc}
Let $b$ be a prime. For $m,s,\alpha,d\in \nat$ with $\alpha\ge 2$, and a set of weights $\bsgamma=(\gamma_u)_{u\subseteq I_s}$, let $p\in \FF_b[x]$ and $\bsq=(q_1,\ldots,q_{ds})\in (\FF_b[x])^{ds}$ be found by Algorithm \ref{algorithm:cbc}. From Theorem \ref{theorem:cbc_bound} in which $j_1=s,d_1=d$ now, we have
  \begin{align*}
    B_{\alpha,d,\bsgamma}(\bsq,p) \le & \frac{1}{(b^m-1)^{1/\lambda}}\left[ \sum_{\emptyset \ne u \subseteq I_{s-1}}\gamma_u^{\lambda}G_{\alpha,d,\lambda,d}^{|u|}+G_{\alpha,d,\lambda,d}\sum_{u \subseteq I_{s-1}}\gamma_{u\cup \{s\}}^{\lambda}G_{\alpha,d,\lambda,d}^{|u|} \right]^{1/\lambda} \\
    = &  \frac{1}{(b^m-1)^{1/\lambda}}\left[ \sum_{\emptyset \ne u \subseteq I_s}\gamma_u^{\lambda}G_{\alpha,d,\lambda,d}^{|u|}\right]^{1/\lambda} ,
  \end{align*}
for $1/(2\min(\alpha,d))<\lambda \le 1$. As we cannot achieve the convergence rate of the mean square worst-case error of order $b^{-2\alpha m}$ in $H_{s,\alpha,\bsgamma}$ \cite{Sha63}, our result is optimal when $d\ge \alpha$. 
\end{remark}

\subsection{Fast component-by-component construction}
Here we assume product weights for the sake of simplicity and show how one can apply the fast CBC construction using the fast Fourier transform. The cost of the CBC construction by naive implementation of Algorithm \ref{algorithm:cbc} is at least of $O(dsb^{2m})$ operations, which can be reduced to $O(dsmb^{m})$ operations for the fast CBC construction using the fast Fourier transform.

According to Algorithm \ref{algorithm:cbc}, we choose an irreducible polynomial $p$ with $\deg(p)=m$, set $q_1 =1$ and construct the polynomials $q_2, \ldots, q_{ds}$ inductively in the following way. Assume that $\bsq_{r-1}=(q_1, \ldots, q_{r-1})$ are already found. Let $r-1=d(j_0-1)+d_0$ and $r=d(j_1-1)+d_1$ such that $j_0,d_0,j_1,d_1 \in \nat$ and $0<d_0,d_1\le d$. Then we have $(j_1,d_1)=(j_0+1,1)$ if $d_0=d$, and $(j_1,d_1)=(j_0,d_0+1)$ otherwise. As mentioned in Remark \ref{remark:cbc_compute}, we have to compute
  \begin{align*}
    & B_{\alpha,d,\bsgamma}((\bsq_{r-1}, q^*_r),p) \\
    = & -1+\frac{1}{b^m}\sum_{n=0}^{b^m-1}\prod_{j=1}^{j_1-1}\left[1-\gamma_j \tilde{D}_{\alpha,b,d} +\gamma_j \tilde{D}_{\alpha,b,d}\prod_{l=1}^{d}\left( 1+\chi_{\alpha,d}(y_{n,d(j-1)+l})\right) \right] \\
    & \times \left[1-\gamma_{j_1} \tilde{D}_{\alpha,b,d} +\gamma_{j_1} \tilde{D}_{\alpha,b,d}\prod_{l=1}^{d_1}\left( 1+\chi_{\alpha,d}(y_{n,d(j_1-1)+l})\right) \right] ,
  \end{align*}
for $q^*_r\in R_{b,m}$, where $\{\bsy_0,\ldots,\bsy_{b^m-1}\}\subset [0,1)^r$ is a polynomial lattice point set $P_{b^m,r}((\bsq_{r-1}, q^*_r),p)$.

Here we introduce the following notation
\begin{align*}
     \eta^{(1)}_{n,r-1} := & \prod_{j=1}^{j_1-1} \Big[1-\gamma_j\tilde{D}_{\alpha,b,d} + \gamma_j\tilde{D}_{\alpha,b,d} \prod_{l=1}^d \left( 1+\chi_{\alpha,d}(y_{n,d(j-1)+l})\right) \Big] , \\
     \eta^{(2)}_{n,r-1} := & \prod_{l=1}^{d_1-1} \left( 1 + \chi_{\alpha,d}(y_{n,d(j_1-1)+l})\right) ,
\end{align*}
where the empty product is set to $1$, and
\begin{align*}
     \eta_{n,r-1} := \eta^{(1)}_{n,r-1}\eta^{(2)}_{n,r-1} ,
\end{align*}
for $0\le n< b^m$. It is straightforward to confirm that we have
  \begin{align*}
    & B_{\alpha,d,\bsgamma}((\bsq_{r-1}, q^*_r),p) \\
    = & -1+\frac{1}{b^m}\sum_{n=0}^{b^m-1}\eta^{(1)}_{n,r-1}\left[1-\gamma_{j_1} \tilde{D}_{\alpha,b,d} +\gamma_{j_1} \tilde{D}_{\alpha,b,d}\eta^{(2)}_{n,r-1}\left( 1 + \chi_{\alpha,d}(y_{n,d(j_1-1)+d_1})\right) \right] \\
    = & -1+\frac{1}{b^m}\sum_{n=0}^{b^m-1}\eta^{(1)}_{n,r-1}\left[1-\gamma_{j_1} \tilde{D}_{\alpha,b,d} +\gamma_{j_1} \tilde{D}_{\alpha,b,d}\eta^{(2)}_{n,r-1}\right] \\
    & + \frac{\gamma_{j_1} \tilde{D}_{\alpha,b,d}}{b^m}\left[ \eta_{0,r-1}\chi_{\alpha,d}(0)+\sum_{n=1}^{b^m-1}\eta_{n,r-1}\chi_{\alpha,d}(y_{n,r})\right].
  \end{align*}
Thus in order to find $q^*_r=q_r\in R_{b,m}$ which minimizes $B_{\alpha,d,\bsgamma}((\bsq_{r-1}, q^*_r),p)$ as a function of $q^*_r$, we only need to compute 
\begin{align}\label{eq:fast_cbc}
     \sum_{n=1}^{b^m-1}\eta_{n,r-1}\chi_{\alpha,d}(y_{n,r}) ,
\end{align}
for $q^*_r\in R_{b,m}$. In the following, we show how we can exploit a feature of polynomial lattice point sets constructed according to Algorithm \ref{algorithm:cbc} to apply the fast CBC construction using the fast Fourier transform.

The key feature in Algorithm \ref{algorithm:cbc} is that we choose an irreducible polynomial $p\in \FF_b[x]$. Then there exists a primitive element $g\in R_{b,m}$, which satisfies
\begin{align*}
     \{g^0 \bmod{p}, g^1 \bmod{p}, \ldots, g^{b^m-2} \bmod{p}\} = R_{b,m},
\end{align*}
and $g^{-1} \bmod{p}=g^{b^m-2} \bmod{p}$. When $q_{r+1}=g^i \bmod{p}$, we can rewrite (\ref{eq:fast_cbc}) as
\begin{align*}
     c_i=\sum_{n=1}^{b^m-1}\eta_{n,r-1}\chi_{\alpha,d}\left(v_m\left( \frac{(g^{i-n} \bmod{p})(x)}{p(x)} \right) \right) ,
\end{align*}
for $1\le i<b^m$. 

We now define the following matrix
\begin{align*}
     \Omega_p := \left[ \chi_{\alpha,d}\left(v_m\left( \frac{(g^{i-n} \bmod{p})(x)}{p(x)} \right) \right)\right]_{1\le i,n<b^m} .
\end{align*}
This matrix is indeed circulant, see for example \cite[Chapter~10.3]{DP10}. Let us denote $\bsc= (c_1,\ldots,c_{b^m-1})^\top$ and $\bseta_{r-1}=(\eta_{1,r-1},\ldots,\eta_{b^m-1,r-1})^\top$. Then we have
\begin{align*}
     \bsc=\Omega_p  \bseta_{r-1} .
\end{align*}
Then for an integer $i_0$ ($1\le i_0< b^m$) such that $c_{i_0}\le c_i$ for $1\le i<b^m$, we set $q_r=g^{i_0}\bmod{p}$. After finding $q_r$, we need to update $\eta^{(1)}_{n,r-1}$ and $\eta^{(2)}_{n,r-1}$ as follows. If $d_1=d$,
\begin{align*}
     \left\{ \begin{array}{ll}
     \eta^{(1)}_{n,r} & = \eta^{(1)}_{n,r-1}\left[ 1-\gamma_{j_1}\tilde{D}_{\alpha,b,d} + \gamma_{j_1}\tilde{D}_{\alpha,b,d}\eta^{(2)}_{n,r-1}\left( 1+\chi_{\alpha,d}\left(v_m\left( \frac{n(x)q_r(x)}{p(x)} \right) \right) \right) \right] , \\
     \eta^{(2)}_{n,r} & = 1 . \\
	\end{array} \right.
\end{align*}
Otherwise if $d_1=1,\ldots,d-1$,
\begin{align*}
     \left\{ \begin{array}{ll}
     \eta^{(1)}_{n,r} & = \eta^{(1)}_{n,r-1} , \\
     \eta^{(2)}_{n,r} & = \eta^{(2)}_{n,r-1}\left( 1+\chi_{\alpha,d}\left(v_m\left( \frac{n(x)q_r(x)}{p(x)} \right) \right) \right) . \\
	\end{array} \right.
\end{align*}

Since the matrix $\Omega_p$ is circulant, the matrix vector multiplication $\Omega_p  \bseta_{r-1}$ can be efficiently done in $O(mb^m)$ operations by using the fast Fourier transform as shown in \cite{NC06a,NC06b}, which significantly reduces the computational cost as compared to the naive matrix vector multiplication. As for memory, we only need to store $\eta^{(1)}_{n,r}$ and $\eta^{(1)}_{n,r}$ for $1\le n<b^m$, which requires $O(b^m)$ memory space. Thus, in total, the fast CBC construction using the fast Fourier transform requires the construction cost of $O(dsmb^m)=O(dsN \log N)$ operations using $O(b^m)=O(N)$ memory. Whereas we focus only on product weights here, it is possible to apply the fast CBC construction to the case with another form of weights by minor modifications of the above procedure. We refer to \cite{CKN06,KSS11} for the fast CBC construction of lattice rules for order-dependent weights and POD (product and order-dependent) weights, respectively.

\subsection{Tractability properties}
Finally in this section, we briefly discuss the tractability properties of our algorithm. In the concept of tractability of multivariate integration, we study the dependence of the minimum number of points $N(\epsilon,s)$ on $\epsilon$ and the dimension $s$ such that $\tilde{e}(P_{N(\epsilon,s)},H_{s,\alpha,\bsgamma})\le \epsilon \tilde{e}(P_0,H_{s,\alpha,\bsgamma})$. Given that the number of points is $N=b^{m}$ and that $B_{\alpha,d,\bsgamma}(\bsq,p)$ is a bound on $\tilde{e}^2(P_{b^m},H_{s,\alpha,\bsgamma})$, we have from the inequality in Remark \ref{remark:cbc}
  \begin{align*}
     N(\epsilon,s) \le \inf_{m\in \nat}\left\{ b^m: \exists \lambda\in \left(\frac{1}{2\min(\alpha,d)},1\right], \; \frac{1}{(b^m-1)^{1/\lambda}}\left[ \sum_{\emptyset \ne u \subseteq I_s}\gamma_u^{\lambda}G_{\alpha,d,\lambda,d}^{|u|}\right]^{1/\lambda}\le \epsilon^2 \gamma_{\emptyset} \right\}  ,
  \end{align*}
Hence, it is already obvious that $N(\epsilon,s)$ depends polynomially on $\epsilon^{-1}$. As for the dependence on the dimension, we have the following corollary. Since the proof is almost the same as that of \cite[Theorem~5.2]{DP07}, we omit it.

\begin{corollary}
Let $\bsgamma=(\gamma_u)_{u\subset \nat}$ be a sequence of weights. We have the following
\begin{enumerate}
\item If there exists $\lambda\in (1/(2\min(\alpha,d)),1]$ such that we have
  \begin{align*}
     \lim_{s\to \infty}\left[ \sum_{\emptyset \ne u\subseteq I_s}\gamma_u^{\lambda}G_{\alpha,d,\lambda,d}^{|u|}\right] < \infty ,
  \end{align*}
then $N(\epsilon,s)$ is bounded above independently of the dimension.

\item If there exists $\lambda\in (1/(2\min(\alpha,d)),1]$ and $q>0$ such that we have
  \begin{align*}
     \limsup_{s\to \infty}\left[\frac{1}{s^q}\sum_{\emptyset \ne u\subseteq I_{s}}\gamma_u^{\lambda}G_{\alpha,d,\lambda,d}^{|u|}\right] < \infty ,
  \end{align*}
then $N(\epsilon,s)$ depends polynomially on the dimension.
\end{enumerate}
\end{corollary}

\section{Numerical experiments}
\label{exp}

We conclude this paper with numerical experiments. In our experiments, the base $b$ is fixed at 2 and product weights are considered. As competitors, we employ higher order digital nets which are constructed by using the first $2^m$ points of Sobol' sequences and Niederreiter-Xing sequences, instead of polynomial lattice point sets, as interlaced components in Definition \ref{def:interlacing_polynomial_lattice}. We use Sobol' sequences as implemented in \cite{JK03} for any $ds$ and Niederreiter-Xing sequences as implemented in \cite{Pir02} as long as $4\le ds\le 16$.

In order to verify the usefulness of our constructed point sets we conduct two types of experiments. The first experiment compares the values of the quality criterion $B_{\alpha,d,\bsgamma}(\bsq,p)$ with the values of the quality criterion for the competitors. The second experiment compares the actual performance of our constructed point sets with those of the competitors using a test function. We also compare the performance of our constructed point sets with that of interlaced scrambled polynomial lattice point sets \cite{GDxx} to see the difference of randomization algorithms.

\subsection{Comparison of the quality criterion}
Here we present the values of $B_{\alpha,d,\bsgamma}(\bsq,p)$ for different choices of $m,s,\alpha,d$ and $\bsgamma$, where $p\in \FF_b[x]$ and $\bsq\in (\FF_b[x])^{ds}$ are found by Algorithm \ref{algorithm:cbc}. For the competitors, we denote the bound on the mean square worst-case error by
  \begin{align*}
    B_{\alpha,d,\bsgamma}(C_1,\ldots,C_{ds}) = \sum_{\emptyset \ne w\subseteq I_{ds}}\gamma_{\phi(w)} \tilde{D}_{\alpha,d}^{|\phi(w)|}\sum_{\substack{\bsl_w\in \nat^{|w|} \\ (\bsl_w,\bszero)\in D^{\perp}(C_1,\ldots,C_{ds})}}r_{\alpha,d}(\bsl_w) ,
  \end{align*}
where $C_1,\ldots,C_{ds}$ are generating matrices of a digital net and $D^{\perp}(C_1,\ldots,C_{ds})$ is its dual net. See \cite{DP10,Nie92a} for what generating matrices means here.

As shown in \cite[Theorem~30]{BD09} combined with \cite[Theorem~12]{BD09}, these competitors have an explicit upper bound on the mean square worst-case error in $H_{s,\alpha,\bsgamma}$, which achieves the optimal rate of convergence. Although not shown here, however, this explicit bound yields a much larger value than $B_{\alpha,d,\bsgamma}(C_1,\ldots,C_{ds})$. Furthermore, it is expected that the values of $B_{\alpha,d,\bsgamma}(C_1,\ldots,C_{ds})$ for the competitors are small enough for the following reason: the first $2^m$ points of Sobol' sequences and Niederreiter-Xing sequences are generally digital $(t,m,s)$-nets with small $t$-value, yielding the large minimum-weight $\rho(C_1,\ldots,C_s)$, which is defined as
  \begin{align*}
    \rho(C_1,\ldots,C_s) = \min_{\bsl\in D^{\perp}(C_1,\ldots,C_s)\setminus \{\bszero\}}\mu_1(\bsl).
  \end{align*}
As can be seen from the definition of $B_{\alpha,d,\bsgamma}(C_1,\ldots,C_{ds})$, our bound on the mean square worst-case error is written in terms of the weight $\mu_1(\bsl)$ for $\bsl\in D^{\perp}(C_1,\ldots,C_{ds})\setminus \{\bszero\}$. Thus, $B_{\alpha,d,\bsgamma}(C_1,\ldots,C_{ds})$ is expected to be small for the competitors, and our comparison here is reasonable in this sense.

We compare the values of $B_{\alpha,d,\bsgamma}(\bsq,p)$ with the values of $B_{\alpha,d,\bsgamma}(C_1,\ldots,C_{ds})$ for the low-dimensional cases in Tables \ref{tb:1}-\ref{tb:4}. In these tables, our constructed point sets based on polynomial lattice point sets are denoted by PLPS for short, and similarly, point sets based on Sobol' sequences and Niederreiter-Xing sequences are respectively denoted by Sobol' and N-X for short. In Tables \ref{tb:1} and \ref{tb:2}, we consider $\gamma_j=1$ for $1\le j\le s$, that is the so-called unweighted case. In Table \ref{tb:1}, we fix $\alpha=d=2$ and change the dimension from $s=1$ to $s=5$. In Table \ref{tb:2}, we fix $s=3$ and change $\alpha$ and $d$ simultaneously. In Tables \ref{tb:3} and \ref{tb:4}, we do similar comparisons for the case $\gamma_j=j^{-2}$ for $1\le j\le s$. In most cases, PLPS outperforms both Sobol' and N-X.

We also consider higher-dimensional cases. In Tables \ref{tb:5} and \ref{tb:6}, we compare the values of $B_{\alpha,d,\bsgamma}(\bsq,p)$ with the values of $B_{\alpha,d,\bsgamma}(C_1,\ldots,C_{ds})$ for $s=10,20,50$, where we fix $\alpha=d=2$. We consider the weights $\gamma_j=1$ and $\gamma_j=j^{-2}$ for $1\le j\le s$ in Tables \ref{tb:5} and \ref{tb:6}, respectively. Although PLPS and Sobol' are comparable for the unweighted case, PLPS outperforms Sobol' for the weighted case.

\subsection{Actual performance for integration}\label{subsec:test}
Finally, we compare the actual performance of our constructed point sets with those of the competitors and interlaced scrambled polynomial lattice point sets using a test function. We consider the function
  \begin{align*}
    f(x_1,\ldots,x_s) = \left( 1+\sum_{j=1}^{s}\frac{x_j}{j^2}\right)^{-1} ,
  \end{align*}
where $0\le x_j< 1$ for all $1\le j\le s$. This function is a simplified model of an elliptic partial differential equation with random coefficients \cite{KSS11}.

For a point set $P_N=\{\bsx_0,\ldots,\bsx_{N-1}\}\subset [0,1)^s$, we measure the performance of $P_N$ by using the root mean square error with respect to a random digital shift. The error estimation is done in the same way as that in \cite[Section~2.9]{DKS13}: We generate $r$ independent random digital shift $\bssigma_1,\ldots,\bssigma_r$ from the uniform distribution on $[0,1)^s$. For $l=1,\ldots,r$, we compute the approximation of $I(f)$ by 
  \begin{align*}
    Q(f;P_{N,\bssigma_l})=\frac{1}{N}\sum_{n=0}^{N-1}f(\bsx_n\oplus \bssigma_l) .
  \end{align*}
Then we take the average
  \begin{align*}
    \bar{Q}(f;P_N)=\frac{1}{r}\sum_{l=1}^{r}Q(f;P_{N,\bssigma_l})  ,
  \end{align*}
which is the final approximation to the integral. An unbiased estimator for the root mean square error of $\bar{Q}(f;P_N)$ is given by
  \begin{align*}
    \mathrm{rmse}(f;P_N):=\sqrt{\frac{1}{r(r-1)}\sum_{l=1}^{r}\left( Q(f;P_{N,\bssigma_l})-\bar{Q}(f;P_N)\right)^2}  .
  \end{align*}

For interlaced scrambled polynomial lattice point sets, we measure their performance by using the root mean square error with respect to a random scrambling. Here a random scrambling is applied to polynomial lattice point sets first and then the resulting point sets are digitally interlaced, see \cite{GDxx}. The error estimation can be done in the same way as above.

In our experiment, we set $r=50$. For every choice of $m$ and $s$, $p\in \FF_b[x]$ and $\bsq\in (\FF_b[x])^{ds}$ are found by Algorithm \ref{algorithm:cbc}, where we set $\alpha=d=2$ and consider the product weights $\gamma_j=j^{-2}$ for $1\le j\le s$. Then our point set $P_{2^m}$ is constructed according to Definition \ref{def:interlacing_polynomial_lattice}. The same point set is used for interlaced scrambled polynomial lattice point sets. Regarding the competitors, the point set $P_{2^m}$ is constructed  using the first $2^m$ points of Sobol' sequences and Niederreiter-Xing sequences as interlaced components in Definition \ref{def:interlacing_polynomial_lattice}, where we consider $d=2$.

We compare the values of $\mathrm{rmse}(f;P_{2^m})$ for our constructed point sets, the competitors and interlaced scrambled polynomial lattice point sets in Tables \ref{tb:7}--\ref{tb:9} for several choices of $s$. In these tables, interlaced scrambled polynomial lattice point sets are denoted by PLPS-sc for short. Regardless of the low-dimensional and higher-dimensional cases, PLPS often outperforms Sobol'. PLPS and N-X are comparable for $s=2$, while PLPS often outperforms N-X for $s=5$. These results indicate the usefulness of our constructed point sets. Moreover, while PLPS-sc outperforms PLPS for $s=1$, PLPS-sc and PLPS are comparable for larger $s$. Thus, for high-dimensional numerical integration, randomization by a digital shift can be used in place of scrambling as a much cheaper way to obtain some statistical estimate on the integration error.

\section*{Acknowledgement}
The author would like to thank two anonymous referees for their careful reading of the manuscript and their many helpful comments and suggestions.

\appendix
\section{Proofs of Theorem \ref{thm:upper_bound} and Corollary \ref{cor:upper_bound}}\label{app:a}

Following \cite{Dic11}, we define a digit interlacing function of non-negative integers. For $d\in \nat$ and $(l_1,\ldots, l_d)\in \nat_0^d$, we denote the $b$-adic expansion of $l_j$ by $l_j=\kappa_{j,0}+\kappa_{j,1}b+\cdots $ for $1\le j\le d$. Then a digit interlacing function $\Ecal_d: \nat_0^d\to \nat_0$ is defined by
  \begin{align}\label{eq:integer_interlace}
    \Ecal_d(l_1,\ldots, l_d) := \sum_{a=0}^{\infty}\sum_{j=1}^{d}\kappa_{j,a}b^{ad+j-1} .
  \end{align}
We note that the above sum is actually a finite sum since the $b$-adic expansions of $l_1,\ldots, l_d$ are finite. We extend the function $\Ecal_d$ to vectors as
  \begin{align*}
    \Ecal_d(l_1,\ldots, l_{ds})=(\Ecal_d(l_1,\ldots, l_d),\ldots, \Ecal_d(l_{d(s-1)+1},\ldots, l_{ds})) ,
  \end{align*}
for $(l_1,\ldots, l_{ds})\in \nat_0^{ds}$. It is straightforward to confirm that $\Ecal_d$ is a bijection. We have the following lemma, which has already appeared in \cite[Section~2.3]{Dic11}.

\begin{lemma}\label{lem:interlace_integer}
For $d\in \nat$, let $(l_1,\ldots, l_d)\in \nat_0^d$ and $(y_1,\ldots,y_d)\in [0,1)^d$. We have
  \begin{align*}
    \wal_{\Ecal_d(l_1,\ldots, l_d)}(\Dcal_d(y_1,\ldots,y_d)) = \prod_{j=1}^{d}\wal_{l_j}(y_j) .
  \end{align*}
\end{lemma}
\begin{proof}
We denote the $b$-adic expansion of $l_j$ by $l_j=\kappa_{j,0}+\kappa_{j,1}b+\cdots $, and the $b$-adic expansion of $y_j$ by $y_j=\eta_{j,1}b^{-1}+\eta_{j,2}b^{-2}+\cdots$ for $1\le j\le d$. Then we have
  \begin{align*}
    \Ecal_d(l_1,\ldots, l_d) = \sum_{a=0}^{\infty}\sum_{j=1}^{d}\kappa_{j,a}b^{ad+j-1} ,
  \end{align*}
and 
  \begin{align*}
    \Dcal_d(y_1,\ldots, y_d) = \sum_{a=1}^{\infty}\sum_{j=1}^{d}\frac{\eta_{j,a}}{b^{j+(a-1)d}} .
  \end{align*}
Thus, from the definition of Walsh functions, we have
  \begin{align*}
    \wal_{\Ecal_d(l_1,\ldots, l_d)}(\Dcal_d(y_1,\ldots,y_d)) & = \omega_b^{\sum_{a=0}^{\infty}\sum_{j=1}^{d}\kappa_{j,a}\eta_{j,a+1}}\\
    & = \prod_{j=1}^{d}\omega_b^{\sum_{a=0}^{\infty}\kappa_{j,a}\eta_{j,a+1}} = \prod_{j=1}^{d}\wal_{l_j}(y_j) ,
  \end{align*}
which completes the proof.
\end{proof}

Using Lemma \ref{lem:interlace_integer}, we have the following result, which bridges between a higher order digital net constructed according to Definition \ref{def:interlacing_polynomial_lattice} and Walsh functions.
\begin{lemma}\label{lem:dual_net_interlace}
Let $b$ be a prime and let $m, s, d \in \nat$. Let $P_{b^m,ds}(\bsq,p)=\{\bsy_0,\ldots,\bsy_{b^m-1}\}\subset [0,1)^{ds}$ be a polynomial lattice point set and let $D^{\perp}(\bsq,p)$ be its dual net. Further, let $P_{b^m}=\{\bsx_0,\ldots,\bsx_{b^m-1}\}\subset [0,1)^{s}$ be a higher order digital net constructed according to Definition \ref{def:interlacing_polynomial_lattice}. Then we have
  \begin{align*}
     \frac{1}{b^m}\sum_{n=0}^{b^m-1}\wal_{\bsk}(\bsx_n)=\left\{ \begin{array}{ll}
     1 & \mbox{if there exists}\ \bsl\in D^{\perp}(\bsq,p)\ \mbox{such that}\ \bsk= \Ecal_d(\bsl), \\
     0 & \mbox{otherwise} . \\
     \end{array} \right.
  \end{align*}
\end{lemma}
\begin{proof}
Let $\bsk\in \nat_0^s$. Since $\Ecal_d$ is a bijection, there exists exactly one $ds$-dimensional vector of integers $\bsl=(l_1,\ldots,l_{ds})\in \nat_0^{ds}$ such that $\Ecal_d(\bsl)=\bsk$. For this $\bsl$, we have
  \begin{align*}
     \frac{1}{b^m}\sum_{n=0}^{b^m-1}\wal_{\bsk}(\bsx_n) = & \frac{1}{b^m}\sum_{n=0}^{b^m-1}\wal_{\Ecal_d(\bsl)}(\Dcal_d(\bsy_n)) \\
     = & \frac{1}{b^m}\sum_{n=0}^{b^m-1}\wal_{\bsl}(\bsy_n) \\
     = & \left\{ \begin{array}{ll}
     1 & \mbox{if}\ \bsl\in D^{\perp}(\bsq,p) , \\
     0 & \mbox{otherwise} , \\
     \end{array} \right.
  \end{align*}
where we use Lemmas \ref{lem:interlace_integer} and \ref{lemma:dual_walsh} in the second and third equalities, respectively. Hence the result follows.
\end{proof}

Let $\emptyset \ne w\subseteq I_{ds}$ and $\bsl_w\in \nat^{|w|}$. In the following lemma, we give a lower bound on $\mu_{\alpha}(\Ecal_d(\bsl_w,\bszero))$.

\begin{lemma}\label{lemma:weight}
For $\emptyset \ne w\subseteq I_{ds}$ and $\bsl_w\in \nat^{|w|}$, we have
  \begin{align}\label{eq:interlace_weight}
    \mu_{\alpha}(\Ecal_d(\bsl_w,\bszero)) \ge \min(\alpha,d)\sum_{r\in w}\mu_1(l_r)+\frac{1}{2}\alpha|w|-\frac{1}{2}\alpha(2d-1)|\phi(w)| ,
  \end{align}
where $\phi$ is defined as in Theorem \ref{thm:upper_bound}.
\end{lemma}

\begin{proof}
From the definition of $\phi$, we have $w\cap\{d(j-1)+1,\ldots, dj\}=\emptyset$ for $j\in I_s\setminus \phi(w)$. By denoting by $w_j$ the index set $w\cap\{d(j-1)+1,\ldots, dj\}$ for $j\in \phi(w)$, it is obvious that we have
  \begin{align*}
    \sum_{r\in w}\mu_1(l_r) = \sum_{j\in \phi(w)}\sum_{r\in w_j}\mu_1(l_r) ,
  \end{align*}
and
  \begin{align*}
    |w|=\sum_{j\in \phi(w)}|w_j| .
  \end{align*}
Using these equalities, the right-hand side of (\ref{eq:interlace_weight}) can be written as
  \begin{align*}
    & \min(\alpha,d)\sum_{r\in w}\mu_1(l_r)+\frac{1}{2}\alpha|w|-\frac{1}{2}\alpha(2d-1)|\phi(w)| \\
    = & \sum_{j\in \phi(w)}\left[ \min(\alpha,d)\sum_{r\in w_j}\mu_1(l_r)+\frac{1}{2}\alpha|w_j|-\frac{1}{2}\alpha(2d-1)\right] .
  \end{align*}

For $j\in \phi(w)$, we denote by $(\bsl_{w_j},\bszero)$ the $d$-dimensional vector with indices $\{d(j-1)+1,\ldots, dj\}$ whose $r$-th component is $l_r$ if $r\in w_j$ and $0$ if $r\in \{d(j-1)+1,\ldots, dj\}\setminus w_j$. Then, the left-hand side of (\ref{eq:interlace_weight}) becomes
  \begin{align*}
    \mu_{\alpha}(\Ecal_d(\bsl_w,\bszero)) = \sum_{j\in \phi(w)}\mu_{\alpha}(\Ecal_d(\bsl_{w_j},\bszero)) .
  \end{align*}
Thus, in order to prove this lemma, it suffices to prove
  \begin{align*}
    \mu_{\alpha}(\Ecal_d(\bsl_{w_j},\bszero)) \ge \min(\alpha,d)\sum_{r\in w_j}\mu_1(l_r)+\frac{1}{2}\alpha|w_j|-\frac{1}{2}\alpha(2d-1) ,
  \end{align*}
for $j\in \phi(w)$. Therefore, we focus on proving the last inequality below.

For $1\le r\le d$ such that $d(j-1)+r\in w_j$, we denote the $b$-adic expansion of $l_{d(j-1)+r}$ by $l_{d(j-1)+r}=\kappa_{r,0}+\kappa_{r,1}b+\cdots +\kappa_{r,\beta-1}b^{\beta-1}$ where $\kappa_{r,\beta-1}\ne 0$. From the definition of $\mu_1$ as in (\ref{eq:dick_weight}) where we set $\alpha=1$, we obtain $\beta=\mu_1(l_{d(j-1)+r})$. Thus we have
  \begin{align*}
    \Ecal_d(\bsl_{w_j},\bszero) = & \sum_{\substack{1\le r\le d\\ d(j-1)+r\in w_j}}\sum_{a=0}^{\mu_1(l_{d(j-1)+r})-1}\kappa_{r,a}b^{ad+r-1} \\
    \ge & \sum_{\substack{1\le r\le d\\ d(j-1)+r\in w_j}}\kappa_{r,\mu_1(l_{d(j-1)+r})-1}b^{d(\mu_1(l_{d(j-1)+r})-1)+r-1} ,
  \end{align*}
where the last inequality is obtained by considering only the term with $a=\mu_1(l_{d(j-1)+r})-1$ in the inner sum.
From the definition of $\mu_{\alpha}$, it is obvious that
  \begin{align}\label{eq:interlace_integer2}
    \mu_{\alpha}\left( \Ecal_d(\bsl_{w_j},\bszero)\right) \ge \mu_{\alpha}\left(\sum_{\substack{1\le r\le d\\ d(j-1)+r\in w_j}}\kappa_{r,\mu_1(l_{d(j-1)+r})-1}b^{d(\mu_1(l_{d(j-1)+r})-1)+r-1}\right) .
  \end{align}

Let us consider the case $|w_j|\le \alpha$ first. Since the sum on the right-hand side of the inequality (\ref{eq:interlace_integer2}) consists of $|w_j|$ terms, we have
  \begin{align*}
    \mu_{\alpha}(\Ecal_d(\bsl_{w_j},\bszero)) & \ge \mu_{\alpha}\left(\sum_{\substack{1\le r\le d\\ d(j-1)+r\in w_j}}\kappa_{r,\mu_1(l_{d(j-1)+r})-1}b^{d(\mu_1(l_{d(j-1)+r})-1)+r-1} \right) \\
    & = \sum_{\substack{1\le r\le d\\ d(j-1)+r\in w_j}}\left[ d(\mu_1(l_{d(j-1)+r})-1)+r \right] \\
    & = d\sum_{r\in w_j}\left[\mu_1(l_r) - 1\right] +\sum_{\substack{1\le r\le d\\ d(j-1)+r\in w_j}}r \\
    & \ge d\sum_{r\in w_j}\mu_1(l_r) - d|w_j|+\sum_{1\le r\le |w_j|}r \\
    & = d\sum_{r\in w_j}\mu_1(l_r) - d|w_j|+\frac{1}{2}|w_j|(|w_j|+1) \\
    & \ge \min(\alpha,d) \sum_{r\in w_j}\mu_1(l_r) - \alpha d+\frac{1}{2}\alpha (|w_j|+1) .
  \end{align*}
Let us consider the case $\alpha < |w_j|$ next. In this case, the sum on the right-hand side of the inequality (\ref{eq:interlace_integer2}) contains more than $\alpha$ terms. Therefore, by using an averaging argument, we obtain
  \begin{align*}
    \mu_{\alpha}(\Ecal_d(\bsl_{w_j},\bszero)) & \ge \mu_{\alpha}\left(\sum_{\substack{1\le r\le d\\ d(j-1)+r\in w_j}}\kappa_{r,\mu_1(l_{d(j-1)+r})-1}b^{d(\mu_1(l_{d(j-1)+r})-1)+r-1} \right) \\
    & \ge \frac{\alpha}{|w_j|}\sum_{\substack{1\le r\le d\\ d(j-1)+r\in w_j}}\left[ d(\mu_1(l_{d(j-1)+r})-1)+r \right] \\
    & = \frac{\alpha d}{|w_j|}\sum_{r\in w_j}\left[\mu_1(l_r) - 1\right] +\frac{\alpha}{|w_j|}\sum_{\substack{1\le r\le d\\ d(j-1)+r\in w_j}}r \\
    & \ge \alpha \sum_{r\in w_j}\mu_1(l_r) - \alpha d +\frac{\alpha}{|w_j|}\sum_{1\le r\le |w_j|}r \\
    & \ge \min(\alpha,d) \sum_{r\in w_j}\mu_1(l_r) - \alpha d+\frac{1}{2}\alpha (|w_j|+1) .
  \end{align*}
Putting the two cases above together, the result follows.
\end{proof}

We are now ready to prove Theorem \ref{thm:upper_bound} and Corollary \ref{cor:upper_bound}.

\begin{proof}[Proof of Theorem \ref{thm:upper_bound}]
Since $\Ecal_d$ is a bijection, the sum over the $|u|$-dimensional vectors $\bsk_u \in \nat^{|u|}$ in (\ref{eq:mse2}) is equal to the sum over the $d|u|$-dimensional vectors $\hat{\bsl}_u=(l_{d(j-1)+1},\cdots,l_{dj})_{j\in u}\in (\nat_0^{d}\setminus \{\bszero\})^{|u|}$, where $\bszero$ is the vector of $d$ zeros, by replacing $\bsk_u$ by $\Ecal_d(\hat{\bsl}_u)=(\Ecal_d(l_{d(j-1)+1},\cdots,l_{dj}))_{j\in u}$. That is, we have
  \begin{align*}
    & \tilde{e}^2(P_N,H_{s,\alpha,\bsgamma}) \\
    \le & \sum_{\emptyset \ne u\subseteq I_s}\gamma_u D_{\alpha,b}^{|u|}\sum_{\hat{\bsl}_u\in (\nat_0^{d}\setminus \{\bszero\})^{|u|}}b^{-2\mu_{\alpha}(\Ecal_d(\hat{\bsl}_u))}\left| \frac{1}{N^2}\sum_{n,n'=0}^{N-1}\wal_{(\Ecal_d(\hat{\bsl}_u), \bszero)}(\bsx_n\ominus \bsx_{n'})\right| .
  \end{align*}

We now set $N=b^m$ and consider a point set $P_{b^m}$ constructed by using a polynomial lattice point set $P_{b^m,ds}(\bsq,p)=\{\bsy_0,\ldots,\bsy_{b^m-1}\}$ according to Definition \ref{def:interlacing_polynomial_lattice}. For $u\subseteq I_{s}$ and $\hat{\bsl}_u\in (\nat_0^d\setminus \{\bszero\})^{|u|}$, we denote by $(\hat{\bsl}_u,\bszero)$ the $ds$-dimensional vector whose $r$-th component is $l_r$ if there exists $j\in u$ such that $r\in \{d(j-1)+1,\ldots,dj\}$ and $0$ otherwise. Then we have
  \begin{align*}
    & \tilde{e}^2(P_{b^m},H_{s,\alpha,\bsgamma}) \\
    \le & \sum_{\emptyset \ne u\subseteq I_s}\gamma_u D_{\alpha,b}^{|u|}\sum_{\hat{\bsl}_u\in (\nat_0^{d}\setminus \{\bszero\})^{|u|}}b^{-2\mu_{\alpha}(\Ecal_d(\hat{\bsl}_u))}\left| \frac{1}{b^{2m}}\sum_{n,n'=0}^{b^m-1}\wal_{\Ecal_d(\hat{\bsl}_u, \bszero)}(\Dcal_d(\bsy_n)\ominus \Dcal_d(\bsy_{n'}))\right| \\
    = & \sum_{\emptyset \ne u\subseteq I_s}\gamma_u D_{\alpha,b}^{|u|}\sum_{\hat{\bsl}_u\in (\nat_0^{d}\setminus \{\bszero\})^{|u|}}b^{-2\mu_{\alpha}(\Ecal_d(\hat{\bsl}_u))} \\
    & \times \left|\frac{1}{b^{2m}}\sum_{n,n'=0}^{b^m-1}\wal_{\Ecal_d(\hat{\bsl}_u, \bszero)}(\Dcal_d(\bsy_n)) \overline{\wal_{\Ecal_d(\hat{\bsl}_u, \bszero)}(\Dcal_d(\bsy_{n'}))}\right| \\
    = & \sum_{\emptyset \ne u\subseteq I_s}\gamma_u D_{\alpha,b}^{|u|}\sum_{\hat{\bsl}_u\in (\nat_0^{d}\setminus \{\bszero\})^{|u|}}b^{-2\mu_{\alpha}(\Ecal_d(\hat{\bsl}_u))} \\
    & \times \left|\frac{1}{b^m}\sum_{n=0}^{b^m-1}\wal_{\Ecal_d(\hat{\bsl}_u, \bszero)}(\Dcal_d(\bsy_n)) \overline{\frac{1}{b^m}\sum_{n'=0}^{b^m-1}\wal_{\Ecal_d(\hat{\bsl}_u, \bszero)}(\Dcal_d(\bsy_{n'}))}\right| \\
    = & \sum_{\emptyset \ne u\subseteq I_s}\gamma_u D_{\alpha,b}^{|u|}\sum_{\substack{\hat{\bsl}_u\in (\nat_0^{d}\setminus \{\bszero\})^{|u|}\\ (\hat{\bsl}_u, \bszero)\in D^{\perp}(\bsq,p)}}b^{-2\mu_{\alpha}(\Ecal_d(\hat{\bsl}_u))} ,
  \end{align*}
where we use the property of Walsh functions in \cite[Proposition~A.6]{DP10} in the first equality and use Lemma \ref{lem:dual_net_interlace} in the third equality. Using the mapping $\phi: I_{ds}\to I_s$, we further have
  \begin{align*}
   \tilde{e}^2(P_{b^m},H_{s,\alpha,\bsgamma}) \le & \sum_{\emptyset \ne u\subseteq I_s}\gamma_u D_{\alpha,b}^{|u|}\sum_{\substack{\emptyset \ne w\subseteq I_{ds}\\ \phi(w)=u}}\sum_{\substack{\bsl_w\in \nat^{|w|}\\ (\bsl_w, \bszero)\in D^{\perp}(\bsq,p)}}b^{-2\mu_{\alpha}(\Ecal_d(\bsl_w, \bszero))} \\
    = & \sum_{\emptyset \ne w\subseteq I_{ds}}\gamma_{\phi(w)} D_{\alpha,b}^{|\phi(w)|}\sum_{\substack{\bsl_w\in \nat^{|w|}\\ (\bsl_w,\bszero)\in D^{\perp}(\bsq,p)}}b^{-2\mu_{\alpha}(\Ecal_d(\bsl_w,\bszero))} ,
  \end{align*}
where we denote by $(\bsl_w,\bszero)$ the $ds$-dimensional vector whose $r$-th component is $l_r$ if $r\in w$ and $0$ otherwise. Finally by using Lemma \ref{lemma:weight}, we have
  \begin{align*}
   & \tilde{e}^2(P_{b^m},H_{s,\alpha,\bsgamma}) \\
   \le & \sum_{\emptyset \ne w\subseteq I_{ds}}\gamma_{\phi(w)} D_{\alpha,b}^{|\phi(w)|}\sum_{\substack{\bsl_w\in \nat^{|w|}\\ (\bsl_w,\bszero)\in D^{\perp}(\bsq,p)}}b^{-2\min(\alpha,d)\sum_{r\in w}\mu_1(l_r)-\alpha|w|+\alpha(2d-1)|\phi(w)|} \\
   = & \sum_{\emptyset \ne w\subseteq I_{ds}}\gamma_{\phi(w)} \tilde{D}_{\alpha,b,d}^{|\phi(w)|}\sum_{\substack{\bsl_w\in \nat^{|w|}\\ (\bsl_w,\bszero)\in D^{\perp}(\bsq,p)}}r_{\alpha,d}(\bsl_w) ,
  \end{align*}
where we write $\tilde{D}_{\alpha,b,d}=b^{(2d-1)\alpha}D_{\alpha,b}$. Hence the result follows.
\end{proof}

\begin{proof}[Proof of Corollary \ref{cor:upper_bound}]
Due to the property of the dual net $D^{\perp}(\bsq,p)$ as in Lemma \ref{lemma:dual_walsh}, we have
  \begin{align}\label{eq:criterion1}
    B_{\alpha,d,\bsgamma}(\bsq,p) & = \sum_{\emptyset \ne w\subseteq I_{ds}}\gamma_{\phi(w)} \tilde{D}_{\alpha,b,d}^{|\phi(w)|}\sum_{\bsl_w\in \nat^{|w|}}r_{\alpha,d}(\bsl_w) \frac{1}{b^m}\sum_{n=0}^{b^m-1}\wal_{(\bsl_w,\bszero)}(\bsy_n) \nonumber \\
    & = \frac{1}{b^m}\sum_{n=0}^{b^m-1}\sum_{\emptyset \ne w\subseteq I_{ds}}\gamma_{\phi(w)} \tilde{D}_{\alpha,b,d}^{|\phi(w)|}\prod_{j\in w}\sum_{l_j=1}^{\infty}r_{\alpha,d}(l_j) \wal_{l_j}(y_{n,j}) .
  \end{align}
From \cite[Section~2.2]{DP05} we obtain
  \begin{align*}
    \sum_{l=1}^{\infty}r_{\alpha,d}(l) \wal_{l}(y) = \chi_{\alpha,d}(y) ,
  \end{align*}
for $y\in [0,1)$. We arrange (\ref{eq:criterion1}) by collecting the terms associated with a given $u\subseteq I_s$ such that $\phi(w)=u$. We then have
  \begin{align*}
    B_{\alpha,d,\bsgamma}(\bsq,p) & = \frac{1}{b^m}\sum_{n=0}^{b^m-1}\sum_{\emptyset \ne u\subseteq I_{s}}\sum_{\substack{\emptyset \ne w\subseteq I_{ds}\\ \phi(w)=u}}\gamma_{\phi(w)} \tilde{D}_{\alpha,b,d}^{|\phi(w)|}\prod_{j\in w}\chi_{\alpha,d}(y_{n,j}) \\
    & = \frac{1}{b^m}\sum_{n=0}^{b^m-1}\sum_{\emptyset \ne u\subseteq I_{s}}\gamma_{u} \tilde{D}_{\alpha,b,d}^{|u|}\sum_{\substack{\emptyset \ne w\subseteq I_{ds}\\ \phi(w)=u}}\prod_{j\in w}\chi_{\alpha,d}(y_{n,j}) \\
    & = \frac{1}{b^m}\sum_{n=0}^{b^m-1}\sum_{\emptyset \ne u\subseteq I_{s}}\gamma_{u} \tilde{D}_{\alpha,b,d}^{|u|}\prod_{j\in u}\; \sum_{\emptyset \ne w\subseteq \{d(j-1)+1,\ldots,dj\}}\prod_{l\in w}\chi_{\alpha,d}(y_{n,l}) \\
    & = \frac{1}{b^m}\sum_{n=0}^{b^m-1}\sum_{\emptyset \ne u\subseteq I_{s}}\gamma_{u} \tilde{D}_{\alpha,b,d}^{|u|}\prod_{j\in u}\left[ -1+\prod_{l=1}^{d}\left( 1+\chi_{\alpha,d}(y_{n,d(j-1)+l})\right) \right] .
  \end{align*}
Hence the result for the first part follows. In case of product weights, we have a more simplified expression of $B_{\alpha,d,\bsgamma}(\bsq,p)$ as
  \begin{align*}
    B_{\alpha,d,\bsgamma}(\bsq,p) & = \frac{1}{b^m}\sum_{n=0}^{b^m-1}\sum_{\emptyset \ne u\subseteq I_{s}}\prod_{j\in u}\gamma_{j} \tilde{D}_{\alpha,b,d}\left[ -1+\prod_{l=1}^d \left( 1+\chi_{\alpha,d}(y_{n,d(j-1)+l})\right) \right] \\
    & = -1+\frac{1}{b^m}\sum_{n=0}^{b^m-1}\prod_{j=1}^{s}\left[1+\gamma_{j} \tilde{D}_{\alpha,b,d}\left[ -1+\prod_{l=1}^d \left( 1+\chi_{\alpha,d}(y_{n,d(j-1)+l})\right) \right] \right].
  \end{align*}
Hence the result for the second part follows.
\end{proof}

\section{Proof of Theorem \ref{theorem:cbc_bound}}\label{app:b}
In the proof of the theorem, we shall use the following inequality that is sometimes referred to as Jensen's inequality. For a sequence $(a_n)_{n\in \nat}$ of non-negative real numbers, we have
	\begin{align*}
		\left( \sum a_n\right)^{\lambda} \le \sum a_n^{\lambda} ,
	\end{align*}
for $0<\lambda\le 1$. We shall also use the following lemma.

\begin{lemma}\label{lem:weight_sum}
Let $b$ be a prime, let $\alpha,d\in \nat$ with $\alpha\ge 2$, and let $\lambda>1/(2\min(\alpha,d))$ be a real number. Let $r_{\alpha,d}:\nat_0 \to \RR$ be given as in Theorem \ref{thm:upper_bound}.
\begin{enumerate}
\item We have
  \begin{align*}
    \sum_{l=1}^{\infty}r_{\alpha,d}^{\lambda}(l)=\frac{b-1}{b^{\lambda \alpha}(b^{2\lambda \min(\alpha,d)}-b)} .
  \end{align*}
\item For $m\in \nat$, we have
  \begin{align*}
    \sum_{\substack{l=1\\ b^m\mid l}}^{\infty}r_{\alpha,d}^{\lambda}(l)=\frac{b-1}{b^{2 \lambda\min(\alpha,d)m+ \lambda\alpha}(b^{2\lambda \min(\alpha,d)}-b)} .
  \end{align*}
\end{enumerate}
\end{lemma}

\begin{proof}
Let us consider the first part. From the definition of $r_{\alpha,d}$, we have
  \begin{align*}
    \sum_{l=1}^{\infty}r_{\alpha,d}^{\lambda}(l) = & \sum_{a=1}^{\infty}\sum_{l=b^{a-1}}^{b^a-1}r_{\alpha,d}^{\lambda}(l) \\
    = & \frac{1}{b^{\lambda \alpha}}\sum_{a=1}^{\infty}\sum_{l=b^{a-1}}^{b^a-1}b^{-2\lambda \min(\alpha,d)a} \\
    = & \frac{1}{b^{\lambda \alpha}}\sum_{a=1}^{\infty}(b^a-b^{a-1})b^{-2\lambda \min(\alpha,d)a} \\
    = & \frac{b-1}{b^{\lambda \alpha}(b^{2\lambda \min(\alpha,d)}-b)} .
  \end{align*}

Let us move on to the second part. In a way similar to the above proof of the first part, we have
  \begin{align*}
    \sum_{\substack{l=1\\ b^m\mid l}}^{\infty}r_{\alpha,d}^{\lambda}(l) = & \sum_{a=1}^{\infty}\sum_{\substack{l=b^{a-1}\\ b^m\mid l}}^{b^a-1}r_{\alpha,d}^{\lambda}(l) \\
    = & \frac{1}{b^{\lambda \alpha}}\sum_{a=m+1}^{\infty}\sum_{\substack{l=b^{a-1}\\ b^m\mid l}}^{b^a-1}b^{-2\lambda \min(\alpha,d)a} \\
    = & \frac{1}{b^{\lambda \alpha}}\sum_{a=m+1}^{\infty}\frac{1}{b^m}(b^a-b^{a-1})b^{-2\lambda \min(\alpha,d)a} \\
    = & \frac{b-1}{b^{2 \lambda\min(\alpha,d)m+ \lambda\alpha}(b^{2\lambda \min(\alpha,d)}-b)} .
  \end{align*}
\end{proof}

\begin{proof}[Proof of Theorem \ref{theorem:cbc_bound}]
We prove the theorem by induction. For $r=1$, we have $q_1=1$. By using the second part of Lemma \ref{lem:weight_sum} where we consider $\lambda=1$ here, $B_{\alpha,d, \bsgamma}(q_1,p)$ can be calculated as
  \begin{align*}
    B_{\alpha,d,\bsgamma}(q_1,p) & = \gamma_{\{1\}}\tilde{D}_{\alpha,b,d}\sum_{\substack{l_1=1\\ b^m\mid l_1}}^{\infty}r_{\alpha,d}(l_1) \\
    & = \frac{1}{b^{2\min(\alpha,d)m}}\gamma_{\{1\}}\tilde{D}_{\alpha,b,d}\frac{b-1}{b^{\alpha}(b^{2\min(\alpha,d)}-b)} \\
    & \le \frac{1}{(b^m-1)^\lambda}\left[ \gamma_{\{1\}}^{\lambda}G_{\alpha,d,\lambda,1}\right]^{1/\lambda} ,
  \end{align*}
for $1/(2\min(\alpha,d)) < \lambda \le 1$. Hence the result follows.

Next we suppose that the inequality (\ref{eq:cbc_bound}) holds true for $r\ge 1$, where we write $r=d(j_0-1)+d_0$ such that $j_0,d_0\in \nat$ and $d_0\in \{1,\ldots, d\}$. By writing $r+1=d(j_1-1)+d_1$ such that $j_1,d_1\in \nat$ and $d_1\in \{1,\ldots, d\}$, we obtain
  \begin{align*}
    (j_1,d_1) = \left\{ \begin{array}{ll}
     (j_0+1,1) & \mathrm{if}\  d_0=d, \\
     (j_0,d_0+1) & \mathrm{otherwise} . \\
     \end{array} \right.
  \end{align*}
We now consider
  \begin{align}\label{eq:recursion}
    & B_{\alpha,d,\bsgamma}((\bsq_r,q^*_{r+1}),p) \nonumber \\
    = & \sum_{\emptyset \ne w\subseteq I_{r+1}}\gamma_{\phi(w)} \tilde{D}_{\alpha,b,d}^{|\phi(w)|}\sum_{\substack{\bsl_w\in \nat^{|w|} \\ (\bsl_w,\bszero)\in D^{\perp}((\bsq_r,q^*_{r+1}),p)}}r_{\alpha,d}(\bsl_w) \nonumber \\
    = & \sum_{\emptyset \ne w\subseteq I_{r}}\gamma_{\phi(w)} \tilde{D}_{\alpha,b,d}^{|\phi(w)|}\sum_{\substack{\bsl_w\in \nat^{|w|} \\ (\bsl_w,\bszero)\in D^{\perp}(\bsq_r,p)}}r_{\alpha,d}(\bsl_w) \nonumber \\
    & + \sum_{w\subseteq I_{r}}\gamma_{\phi(w\cup\{r+1\})} \tilde{D}_{\alpha,b,d}^{|\phi(w\cup\{r+1\})|}\sum_{\substack{\bsl_{w\cup \{r+1\}}\in \nat^{|w|+1} \\ (\bsl_{w\cup \{r+1\}},\bszero)\in D^{\perp}((\bsq_r,q^*_{r+1}),p)}}r_{\alpha,d}(\bsl_{w\cup \{r+1\}}) \nonumber \\
    = & B_{\alpha,d,\bsgamma}(\bsq_r,p)+\theta(q^*_{r+1}) ,
  \end{align}
where we define 
  \begin{align*}
    \theta(q^*_{r+1}) := \sum_{w\subseteq I_{r}}\gamma_{\phi(w\cup\{r+1\})} \tilde{D}_{\alpha,b,d}^{|\phi(w\cup\{r+1\})|}\sum_{\substack{\bsl_{w\cup \{r+1\}}\in \nat^{|w|+1} \\ (\bsl_{w\cup \{r+1\}},\bszero)\in D^{\perp}((\bsq_r,q^*_{r+1}),p)}}r_{\alpha,d}(\bsl_{w\cup \{r+1\}}) .
  \end{align*}
In order to minimize $B_{\alpha,d,\bsgamma}((\bsq_r,q^*_{r+1}),p)$ as a function of $q^*_{r+1}$, we only need to consider $\theta(q^*_{r+1})$. Due to an averaging argument, the minimal value of $\theta(q^*_{r+1})$ has to be less than or equal to the average value of $\theta(q^*_{r+1})$ over $q^*_{r+1}\in R_{b,m}$. Let $q_{r+1}\in R_{b,m}$ be a minimizer of $\theta$. Applying Jensen's inequality, we have for $0<\lambda\le 1$
  \begin{align*}
    \theta^{\lambda}(q_{r+1}) = & \min_{q^{*}_{r+1}\in R_{b,m}}\theta^{\lambda}(q^*_{r+1}) \\
    \le & \frac{1}{b^m-1}\sum_{q^*_{r+1}\in R_{b,m}}\theta^{\lambda}(q^*_{r+1}) \\
    \le & \frac{1}{b^m-1}\sum_{q^*_{r+1}\in R_{b,m}}\sum_{w\subseteq I_{r}}\gamma_{\phi(w\cup\{r+1\})}^{\lambda} \tilde{D}_{\alpha,b,d}^{\lambda |\phi(w\cup\{r+1\})|} \\
    & \times \sum_{\substack{\bsl_{w\cup\{r+1\}}\in \nat^{|w|+1} \\ (\bsl_{w\cup\{r+1\}},\bszero)\in D^{\perp}((\bsq_r,q^*_{r+1}),p)}}r_{\alpha,d}^{\lambda}(\bsl_{w\cup\{r+1\}}) \\
    = &  \sum_{w\subseteq I_{r}}\gamma_{\phi(w\cup\{r+1\})}^{\lambda} \tilde{D}_{\alpha,b,d}^{\lambda |\phi(w\cup\{r+1\})|} \\
    & \times \frac{1}{b^m-1}\sum_{q^*_{r+1}\in R_{b,m}}\sum_{\substack{\bsl_{w\cup\{r+1\}}\in \nat^{|w|+1} \\ (\bsl_{w\cup\{r+1\}},\bszero)\in D^{\perp}((\bsq_r,q^*_{r+1}),p)}}r_{\alpha,d}^{\lambda}(\bsl_{w\cup\{r+1\}}) .
  \end{align*}
For $w\subseteq I_r$, we have the following on the inner double sum in the last expression. If $l_{r+1}$ is a multiple of $b^m$, we always have $\rtr_m(l_{r+1})=0$ and the condition $(\bsl_{w\cup\{r+1\}},\bszero)\in D^{\perp}((\bsq_r,q^*_{r+1}),p)$ reduces to the equation $\rtr_m(\bsl_w)\cdot \bsq_w=0 \pmod p$. Otherwise if $l_{r+1}$ is not a multiple of $b^m$, we have $\rtr_m(l_{r+1})\ne 0$ and $\rtr_m(l_{r+1})q^*_{r+1}$ cannot be a multiple of $p$ by considering that $p$ is irreducible. Hence we have
  \begin{align}
    & \frac{1}{b^m-1}\sum_{q^*_{r+1}\in R_{b,m}}\sum_{\substack{\bsl_{w\cup\{r+1\}}\in \nat^{|w|+1} \\ (\bsl_{w\cup\{r+1\}},\bszero)\in D^{\perp}((\bsq_r,q^*_{r+1}),p)}}r_{\alpha,d}^{\lambda}(\bsl_{w\cup\{r+1\}}) \nonumber \\
    = & \sum_{\substack{l_{r+1}=1\\ b^m\mid l_{r+1}}}^{\infty}r_{\alpha,d}^{\lambda}(l_{r+1})\sum_{\substack{\bsl_w\in \nat^{|w|}\\ \rtr_m(\bsl_w)\cdot \bsq_w=0 \pmod p}}r_{\alpha,d}^{\lambda}(\bsl_w) \label{eq:cbc_proof1}\\
    & + \frac{1}{b^m-1}\sum_{\substack{l_{r+1}=1\\ b^m\nmid l_{r+1}}}^{\infty}r_{\alpha,d}^{\lambda}(l_{r+1})\sum_{\substack{\bsl_w\in \nat^{|w|}\\ \rtr_m(\bsl_w)\cdot \bsq_w\ne 0 \pmod p}}r_{\alpha,d}^{\lambda}(\bsl_w) \label{eq:cbc_proof2}.
  \end{align}
Here we apply the second and first parts of Lemma \ref{lem:weight_sum} to (\ref{eq:cbc_proof1}) and (\ref{eq:cbc_proof2}), respectively, to obtain
  \begin{align*}
    \sum_{\substack{l_{r+1}=1\\ b^m\mid l_{r+1}}}^{\infty}r_{\alpha,d}^{\lambda}(l_{r+1})=\frac{b-1}{b^{2\lambda\min(\alpha,d)m+\lambda\alpha}(b^{2\lambda\min(\alpha,d)}-b)} \le \frac{C_{\alpha,d,\lambda}}{b^{2\lambda\min(\alpha,d)m}} ,
  \end{align*}
and
  \begin{align*}
   \sum_{\substack{l_{r+1}=1\\ b^m\nmid l_{r+1}}}^{\infty}r_{\alpha,d}^{\lambda}(l_{r+1}) \le & \sum_{l_{r+1}=1}^{\infty}r_{\alpha,d}^{\lambda}(l_{r+1}) \\
   = & \frac{b-1}{b^{\lambda\alpha}(b^{2\lambda\min(\alpha,d)}-b)} \\
   \le & C_{\alpha,d,\lambda} ,
  \end{align*}
for $1/(2\min(\alpha,d)) < \lambda \le 1$. From these inequalities, we have 
  \begin{align*}
    & \frac{1}{b^m-1}\sum_{q^*_{r+1}\in R_{b,m}}\sum_{\substack{\bsl_{w\cup\{r+1\}}\in \nat^{|w|+1} \\ (\bsl_{w\cup\{r+1\}},\bszero)\in D^{\perp}((\bsq_r,q^*_{r+1}),p)}}r_{\alpha,d}^{\lambda}(\bsl_{w\cup\{r+1\}}) \\
    \le & \frac{C_{\alpha,d,\lambda}}{b^{2\lambda\min(\alpha,d)m}}\sum_{\substack{\bsl_w\in \nat^{|w|}\\ \rtr_m(\bsl_w)\cdot \bsq_w=0 \pmod p}}r_{\alpha,d}^{\lambda}(\bsl_w)+\frac{C_{\alpha,d,\lambda}}{b^m-1}\sum_{\substack{\bsl_w\in \nat^{|w|}\\ \rtr_m(\bsl_w)\cdot \bsq_w\ne 0 \pmod p}}r_{\alpha,d}^{\lambda}(\bsl_w) \\
    \le & \frac{C_{\alpha,d,\lambda}}{b^m-1}\sum_{\bsl_w\in \nat^{|w|}}r_{\alpha,d}^{\lambda}(\bsl_w) = \frac{C_{\alpha,d,\lambda}}{b^m-1}\left[\sum_{l=1}^{\infty}r_{\alpha,d}^{\lambda}(l)\right]^{|w|} \le \frac{C_{\alpha,d,\lambda}^{|w|+1}}{b^m-1} ,
  \end{align*}
where we use the first part of Lemma \ref{lem:weight_sum} again in the last inequality. Thus, we have a bound on $\theta^{\lambda}(q_{r+1})$ as
  \begin{align}\label{eq:bound_theta}
    \theta^{\lambda}(q_{r+1}) \le \frac{1}{b^m-1}\sum_{w\subseteq I_{r}}\gamma_{\phi(w\cup\{r+1\})}^{\lambda} \tilde{D}_{\alpha,b,d}^{\lambda |\phi(w\cup\{r+1\})|}C_{\alpha,d,\lambda}^{|w|+1} .
  \end{align}

We arrange the sum on the right-hand side of (\ref{eq:bound_theta}). We recall that $r=d(j_0-1)+d_0=d(j_1-1)+d_1-1$. Let us define two disjoint subsets $J_1:= \{1,\ldots, d(j_1-1) \}$ and $J_2:= \{d(j_1-1)+1,\ldots, d(j_1-1)+d_1-1 \}$. Here $J_1$ is empty for the case $j_1=1$, and $J_2$ is empty for the case $d_1=1$. We have $J_1\cup J_2=\{1,\ldots,d(j_1-1)+d_1-1\}=I_r$. For a subset $w\subseteq I_r$, we write $w_1=w\cap J_1$ and $w_2=w\cap J_2$. Then we have $w_1\cup w_2=w$ and $w_1\cap w_2=\emptyset$. Furthermore, we have $\phi(w\cup\{r+1\})=\phi(w_1)\cup\{j_1\}$. Through this argument, we have
  \begin{align}\label{eq:arrange1}
    & \sum_{w\subseteq I_{r}}\gamma_{\phi(w\cup\{r+1\})}^{\lambda} \tilde{D}_{\alpha,b,d}^{\lambda |\phi(w\cup\{r+1\})|}C_{\alpha,d,\lambda}^{|w|+1} \nonumber \\
    = & \sum_{w_1\subseteq J_1}\sum_{w_2\subseteq J_2}\gamma_{\phi(w_1)\cup \{j_1\}}\tilde{D}_{\alpha,b,d}^{\lambda |\phi(w_1)\cup\{j_1\}|}C_{\alpha,d,\lambda}^{|w_1|+|w_2|+1} \nonumber \\
    = & \left( \sum_{w_1\subseteq J_1}\gamma_{\phi(w_1)\cup \{j_1\}}\tilde{D}_{\alpha,b,d}^{\lambda |\phi(w_1)\cup\{j_1\}|}C_{\alpha,d,\lambda}^{|w_1|}\right) \left( \sum_{w_2\subseteq J_2}C_{\alpha,d,\lambda}^{|w_2|+1}\right) .
  \end{align}
We further arrange the first sum of (\ref{eq:arrange1}) by collecting the terms such that $\phi(w_1)=u$ for $u\subseteq I_{j_1-1}$. For such terms, at least one element from $\{d(j-1)+1,\ldots,dj\}$ for every $j\in u$ must be included. Thus we have
  \begin{align*}
    & \sum_{w_1\subseteq J_1}\gamma_{\phi(w_1)\cup \{j_1\}}\tilde{D}_{\alpha,b,d}^{\lambda |\phi(w_1)\cup\{j_1\}|}C_{\alpha,d,\lambda}^{|w_1|} \\
    = & \sum_{u\subseteq I_{j_1-1}}\gamma_{u\cup \{j_1\}}\tilde{D}_{\alpha,b,d}^{\lambda (|u|+1)}\sum_{\substack{ w_1\subseteq J_1 \\ \phi(w_1)=u}}C_{\alpha,d,\lambda}^{|w_1|} \\
    = & \sum_{u\subseteq I_{j_1-1}}\gamma_{u\cup \{j_1\}}\tilde{D}_{\alpha,b,d}^{\lambda (|u|+1)}\prod_{j\in u}\; \sum_{\emptyset \ne v_j\subseteq \{d(j-1)+1,\ldots,d(j-1)+d\}}C_{\alpha,d,\lambda}^{|v_j|} \\
    = & \sum_{u\subseteq I_{j_1-1}}\gamma_{u\cup \{j_1\}}\tilde{D}_{\alpha,b,d}^{\lambda (|u|+1)}\left[ -1+(1+C_{\alpha,d,\lambda})^d \right]^{|u|} \\
    = & \tilde{D}_{\alpha,b,d}^{\lambda}\sum_{u\subseteq I_{j_1-1}}\gamma_{u\cup \{j_1\}}G_{\alpha,d,\lambda,d}^{|u|} .
  \end{align*}
For the second sum of (\ref{eq:arrange1}) we have
  \begin{align*}
    \sum_{w_2\subseteq J_2}C_{\alpha,d,\lambda}^{|w_2|+1} = C_{\alpha,d,\lambda}(1+C_{\alpha,d,\lambda})^{d_1-1} = (1+C_{\alpha,d,\lambda})^{d_1}-(1+C_{\alpha,d,\lambda})^{d_1-1} .
  \end{align*}
By substituting these results into (\ref{eq:arrange1}), we obtain
  \begin{align*}
    & \sum_{w\subseteq I_{r}}\gamma_{\phi(w\cup\{r+1\})}^{\lambda} \tilde{D}_{\alpha,b,d}^{\lambda |\phi(w\cup\{r+1\})|}C_{\alpha,d,\lambda}^{|w|+1} \\
    = & (G_{\alpha,d,\lambda,d_1}-G_{\alpha,d,\lambda,d_1-1})\sum_{u\subseteq I_{j_1-1}}\gamma_{u\cup \{j_1\}}G_{\alpha,d,\lambda,d}^{|u|} .
  \end{align*}
From (\ref{eq:bound_theta}) we obtain a bound on $\theta^{\lambda}(q_{r+1})$. Finally, by applying Jensen's inequality to (\ref{eq:recursion}), we obtain
  \begin{align*}
    B_{\alpha,d,\bsgamma}^\lambda(\bsq_{r+1},p) = & \left( B_{\alpha,d,\bsgamma}(\bsq_r ,p)+\theta(q_{r+1})\right)^\lambda \\
    \le & B_{\alpha,d,\bsgamma}^{\lambda}(\bsq_r ,p) + \theta^{\lambda}(q_{r+1}) \\
    \le & \frac{1}{b^m-1}\left[ \sum_{\emptyset \ne u \subseteq I_{j_0-1}}\gamma_u^{\lambda}G_{\alpha,d,\lambda,d}^{|u|}+G_{\alpha,d,\lambda,d_0}\sum_{u \subseteq I_{j_0-1}}\gamma_{u\cup \{j_0\}}^{\lambda}G_{\alpha,d,\lambda,d}^{|u|} \right] \\
    & + \frac{1}{b^m-1}(G_{\alpha,d,\lambda,d_1}-G_{\alpha,d,\lambda,d_1-1})\sum_{u\subseteq I_{j_1-1}}\gamma_{u\cup \{j_1\}}G_{\alpha,d,\lambda,d}^{|u|} \\
    = & \frac{1}{b^m-1}\left[ \sum_{\emptyset \ne u \subseteq I_{j_1-1}}\gamma_u^{\lambda}G_{\alpha,d,\lambda,d}^{|u|} +G_{\alpha,d,\lambda,d_1}\sum_{u\subseteq I_{j_1-1}}\gamma_{u\cup \{j_1\}}G_{\alpha,d,\lambda,d}^{|u|}\right] .
  \end{align*}
Hence the result follows.
\end{proof}


\clearpage
\begin{table}
\begin{center}
{\scriptsize
\caption{Values of $B_{\alpha,d,\bsgamma}(\bsq,p)$ and $B_{\alpha,d,\bsgamma}(C_1,\ldots,C_{ds})$ for $\gamma_j=1$, $1\le j\le s$ and $(\alpha,d)=(2,2)$ with three choices of $s=1,2,5$.}
\begin{tabular}{l||ll|lll|lll}
\hline
$m$ & \multicolumn{2}{|c|}{$s=1$} & \multicolumn{3}{|c|}{$s=2$} & \multicolumn{3}{|c}{$s=5$} \\ \hline
  & Sobol' & PLPS & Sobol' & N-X & PLPS & Sobol' & N-X & PLPS \\ \hline
4	& 2.13e-5	& 2.11e-5	& 2.69e-3	& 2.56e-3	& 2.70e-3	& 1.27e+0	& 1.76e+0	& 9.81e-1	\\
5	& 1.51e-6	& 1.42e-6	& 1.06e-3	& 2.62e-4	& 3.05e-4	& 4.05e-1	& 8.52e-1	& 2.91e-1	\\
6	& 1.38e-7	& 9.56e-8	& 9.51e-5	& 2.04e-5	& 7.58e-5	& 1.84e-1	& 1.50e-1	& 7.42e-2	\\
7	& 2.77e-8	& 6.38e-9	& 3.27e-6	& 1.61e-6	& 6.94e-6	& 4.37e-2	& 6.13e-2	& 2.59e-2	\\
8	& 1.11e-8	& 4.24e-10	& 4.34e-7	& 1.74e-7	& 4.82e-7	& 2.15e-2	& 1.84e-2	& 6.55e-3	\\
9	& 5.26e-9	& 2.81e-11	& 4.92e-8	& 3.08e-8	& 8.09e-8	& 1.28e-2	& 5.97e-3	& 1.94e-3	\\
10	& 2.62e-9	& 1.86e-12	& 1.32e-8	& 1.11e-8	& 5.78e-9	& 9.43e-4	& 4.54e-3	& 3.97e-4	\\
11	& 1.31e-9	& 1.24e-13	& 5.59e-9	& 5.27e-9	& 5.39e-10	& 4.59e-4	& 3.74e-3	& 7.42e-5	\\
12	& 6.55e-10	& 6.44e-15	& 2.57e-9	& 2.61e-9	& 4.64e-11	& 1.13e-4	& 5.36e-5	& 1.82e-5	\\
13	& 3.28e-10	& 4.44e-16	& 1.28e-9	& 1.30e-9	& 4.85e-12	& 8.07e-5	& 1.05e-5	& 4.32e-6	\\
14	& 1.64e-10	& $< 10^{-16}$	& 6.37e-10	& 6.52e-10	& 3.99e-13	& 1.38e-5	& 1.62e-6	& 7.18e-7	\\
15	& 8.19e-11	& $< 10^{-16}$	& 3.19e-10	& 3.19e-10	& 4.35e-14	& 7.52e-7	& 2.39e-7	& 1.35e-7	\\
\hline
\end{tabular}
\label{tb:1}}
\end{center}
\end{table}

\begin{table}
\begin{center}
{\scriptsize
\caption{Values of $B_{\alpha,d,\bsgamma}(\bsq,p)$ and $B_{\alpha,d,\bsgamma}(C_1,\ldots,C_{ds})$ for $\gamma_j=1$, $1\le j\le s$ and $s=3$ with two choices of $(\alpha,d)=(2,2),(3,3)$.}
\begin{tabular}{l||lll|lll}
\hline
$m$ & \multicolumn{3}{|c|}{$(\alpha,d)=(2,2)$} & \multicolumn{3}{|c}{$(\alpha,d)=(3,3)$} \\ \hline
  & Sobol' & N-X & PLPS & Sobol' & N-X & PLPS \\ \hline
4	& 7.16e-2	& 2.08e-1	& 4.77e-2	& 1.86e+3	& 1.79e+3	& 1.14e+2	\\
5	& 3.06e-2	& 1.33e-2	& 8.05e-3	& 6.11e+2	& 7.22e+2	& 1.87e+1	\\
6	& 6.18e-3	& 2.92e-3	& 1.90e-3	& 2.58e+2	& 4.06e+2	& 1.14e+1	\\
7	& 9.08e-4	& 1.30e-3	& 2.79e-4	& 2.07e+1	& 1.01e+2	& 1.35e+0	\\
8	& 2.42e-4	& 3.74e-4	& 6.02e-5	& 3.55e+0	& 9.83e+1	& 1.34e-1	\\
9	& 8.86e-6	& 4.40e-6	& 7.53e-6	& 1.80e+0	& 1.65e+0	& 1.74e-2	\\
10	& 1.58e-6	& 7.80e-7	& 9.00e-7	& 2.17e-1	& 1.59e+0	& 2.29e-3	\\
11	& 1.20e-6	& 1.66e-7	& 1.45e-7	& 1.77e-2	& 8.70e-3	& 1.34e-4	\\
12	& 6.41e-8	& 1.69e-8	& 1.61e-8	& 4.04e-3	& 3.50e-3	& 8.42e-6	\\
13	& 7.57e-9	& 4.92e-9	& 3.08e-9	& 1.97e-3	& 1.97e-3	& 8.32e-7	\\
14	& 2.43e-9	& 2.06e-9	& 2.37e-10	& 9.82e-4	& 1.13e-3	& 5.14e-8	\\
15	& 1.20e-9	& 9.80e-10	& 3.18e-11	& 4.91e-4	& 4.79e-4	& 2.75e-9	\\
\hline
\end{tabular}
\label{tb:2}}
\end{center}
\end{table}

\begin{table}
\begin{center}
{\scriptsize
\caption{Values of $B_{\alpha,d,\bsgamma}(\bsq,p)$ and $B_{\alpha,d,\bsgamma}(C_1,\ldots,C_{ds})$ for $\gamma_j=j^{-2}$, $1\le j\le s$ and  $(\alpha,d)=(2,2)$ with three choices of $s=1,2,5$.}
\begin{tabular}{l||ll|lll|lll}
\hline
$m$ & \multicolumn{2}{|c|}{$s=1$} & \multicolumn{3}{|c|}{$s=2$} & \multicolumn{3}{|c}{$s=5$} \\ \hline
  & Sobol' & PLPS & Sobol' & N-X & PLPS & Sobol' & N-X & PLPS \\ \hline
4	& 2.13e-5	& 2.11e-5	& 6.89e-4	& 7.34e-4	& 6.91e-4	& 2.78e-2	& 1.31e-1	& 6.67e-3	\\
5	& 1.51e-6	& 1.42e-6	& 2.66e-4	& 6.76e-5	& 7.72e-5	& 3.33e-3	& 1.04e-1	& 1.38e-3	\\
6	& 1.38e-7	& 9.56e-8	& 2.39e-5	& 5.26e-6	& 1.90e-5	& 6.07e-4	& 5.56e-4	& 3.16e-4	\\
7	& 2.77e-8	& 6.38e-9	& 8.38e-7	& 4.30e-7	& 1.74e-6	& 1.61e-4	& 1.32e-4	& 6.41e-5	\\
8	& 1.11e-8	& 4.24e-10	& 1.17e-7	& 5.71e-8	& 1.21e-7	& 7.98e-5	& 4.98e-5	& 1.46e-5	\\
9	& 5.26e-9	& 2.81e-11	& 1.62e-8	& 1.37e-8	& 2.02e-8	& 1.94e-5	& 1.41e-5	& 2.35e-6	\\
10	& 2.62e-9	& 1.86e-12	& 5.27e-9	& 4.74e-9	& 1.45e-9	& 2.27e-6	& 3.53e-6	& 5.09e-7	\\
11	& 1.31e-9	& 1.24e-13	& 2.38e-9	& 2.30e-9	& 1.35e-10	& 1.44e-6	& 2.30e-6	& 6.98e-8	\\
12	& 6.55e-10	& 6.44e-15	& 1.13e-9	& 1.14e-9	& 1.16e-11	& 5.39e-8	& 1.06e-7	& 1.70e-8	\\
13	& 3.28e-10	& 4.44e-16	& 5.65e-10	& 6.91e-10	& 1.21e-12	& 3.38e-8	& 1.63e-8	& 2.69e-9	\\
14	& 1.64e-10	& $< 10^{-16}$	& 2.82e-10	& 3.45e-10	& 9.97e-14	& 4.33e-9	& 1.57e-9	& 3.92e-10	\\
15	& 8.19e-11	& $< 10^{-16}$	& 1.41e-10	& 1.41e-10	& 1.09e-14	& 1.06e-9	& 3.91e-10	& 7.29e-11	\\
\hline
\end{tabular}
\label{tb:3}}
\end{center}
\end{table}

\begin{table}
\begin{center}
{\scriptsize
\caption{Values of $B_{\alpha,d,\bsgamma}(\bsq,p)$ and $B_{\alpha,d,\bsgamma}(C_1,\ldots,C_{ds})$ for $\gamma_j=j^{-2}$, $1\le j\le s$ and $s=3$ with two choices of $(\alpha,d)=(2,2),(3,3)$.}
\begin{tabular}{l||lll|lll}
\hline
$m$ & \multicolumn{3}{|c|}{$(\alpha,d)=(2,2)$} & \multicolumn{3}{|c}{$(\alpha,d)=(3,3)$} \\ \hline
  & Sobol' & N-X & PLPS & Sobol' & N-X & PLPS \\ \hline
4	& 5.49e-3	& 2.21e-2	& 2.38e-3	& 1.12e+2	& 1.31e+2	& 6.13e+0	\\
5	& 2.03e-3	& 9.73e-4	& 4.25e-4	& 1.86e+1	& 2.42e+1	& 6.03e-1	\\
6	& 2.00e-4	& 2.84e-4	& 9.00e-5	& 7.21e+0	& 1.13e+1	& 3.72e-1	\\
7	& 2.84e-5	& 3.93e-5	& 1.37e-5	& 5.82e-1	& 2.81e+0	& 5.32e-2	\\
8	& 6.95e-6	& 1.14e-5	& 2.21e-6	& 1.03e-1	& 2.73e+0	& 4.58e-3	\\
9	& 2.87e-7	& 2.64e-7	& 2.53e-7	& 5.10e-2	& 4.60e-2	& 5.02e-4	\\
10	& 5.04e-8	& 7.85e-8	& 3.22e-8	& 6.09e-3	& 4.41e-2	& 7.55e-5	\\
11	& 3.61e-8	& 8.29e-9	& 4.35e-9	& 5.20e-4	& 2.67e-4	& 3.98e-6	\\
12	& 3.01e-9	& 1.93e-9	& 5.93e-10	& 1.26e-4	& 1.10e-4	& 2.38e-7	\\
13	& 8.19e-10	& 8.93e-10	& 9.78e-11	& 6.17e-5	& 6.10e-5	& 2.32e-8	\\
14	& 3.72e-10	& 3.79e-10	& 7.46e-12	& 3.08e-5	& 3.39e-5	& 2.02e-9	\\
15	& 1.85e-10	& 2.27e-10	& 1.14e-12	& 1.54e-5	& 1.52e-5	& 1.05e-10	\\
\hline
\end{tabular}
\label{tb:4}}
\end{center}
\end{table}

\begin{table}
\begin{center}
{\scriptsize
\caption{Values of $B_{\alpha,d,\bsgamma}(\bsq,p)$ and $B_{\alpha,d,\bsgamma}(C_1,\ldots,C_{ds})$ for $\gamma_j=1$, $1\le j\le s$ and $(\alpha,d)=(2,2)$ with three choices of $s=10,20,50$.}
\begin{tabular}{l||ll|ll|ll}
\hline
$m$ & \multicolumn{2}{|c|}{$s=10$} & \multicolumn{2}{|c|}{$s=20$} & \multicolumn{2}{|c}{$s=50$} \\ \hline
  & Sobol' & PLPS & Sobol' & PLPS & Sobol' & PLPS \\ \hline
4	& 4.76e+1	& 4.74e+1	& 3.75e+4	& 3.75e+4	& 1.74e+13	& 1.74e+13	\\
5	& 2.33e+1	& 2.32e+1	& 1.87e+4	& 1.87e+4	& 8.70e+12	& 8.70e+12	\\
6	& 1.14e+1	& 1.12e+1	& 9.37e+3	& 9.37e+3	& 4.35e+12	& 4.35e+12	\\
7	& 5.63e+0	& 5.29e+0	& 4.68e+3	& 4.68e+3	& 2.17e+12	& 2.17e+12	\\
8	& 2.78e+0	& 2.41e+0	& 2.34e+3	& 2.34e+3	& 1.09e+12	& 1.09e+12	\\
9	& 1.47e+0	& 1.03e+0	& 1.17e+3	& 1.17e+3	& 5.44e+11	& 5.44e+11	\\
10	& 7.30e-1	& 4.07e-1	& 5.86e+2	& 5.85e+2	& 2.72e+11	& 2.72e+11	\\
11	& 3.94e-1	& 1.78e-1	& 2.93e+2	& 2.92e+2	& 1.36e+11	& 1.36e+11	\\
12	& 1.70e-1	& 6.65e-2	& 1.47e+2	& 1.46e+2	& 6.80e+10	& 6.79e+10	\\
13	& 1.09e-1	& 2.59e-2	& 7.35e+1	& 7.25e+1	& 3.40e+10	& 3.40e+10	\\
14	& 5.26e-2	& 9.49e-3	& 3.67e+1	& 3.61e+1	& 1.70e+10	& 1.70e+10 \\
15	& 3.52e-2	& 3.37e-3	& 1.84e+1	& 1.79e+1	& 8.49e+9	& 8.49e+9 \\
\hline                                                               
\end{tabular}
\label{tb:5}}
\end{center}
\end{table}

\begin{table}
\begin{center}
{\scriptsize
\caption{Values of $B_{\alpha,d,\bsgamma}(\bsq,p)$ and $B_{\alpha,d,\bsgamma}(C_1,\ldots,C_{ds})$ for $\gamma_j=j^{-2}$, $1\le j\le s$ and  $(\alpha,d)=(2,2)$ with three choices of $s=10,20,50$.}
\begin{tabular}{l||ll|ll|ll}
\hline
$m$ & \multicolumn{2}{|c|}{$s=10$} & \multicolumn{2}{|c|}{$s=20$} & \multicolumn{2}{|c}{$s=50$} \\ \hline
  & Sobol' & PLPS & Sobol' & PLPS & Sobol' & PLPS \\ \hline
4	& 3.85e-2	& 1.29e-2	& 4.70e-2	& 1.72e-2	& 5.21e-2	& 2.01e-2	\\
5	& 7.91e-3	& 3.27e-3	& 1.30e-2	& 4.85e-3	& 1.54e-2	& 6.00e-3	\\
6	& 1.24e-3	& 8.65e-4	& 2.56e-3	& 1.41e-3	& 3.93e-3	& 1.85e-3	\\
7	& 5.01e-4	& 2.11e-4	& 8.33e-4	& 3.87e-4	& 1.57e-3	& 5.55e-4	\\
8	& 2.31e-4	& 5.41e-5	& 3.57e-4	& 1.04e-4	& 6.27e-4	& 1.60e-4	\\
9	& 8.81e-5	& 1.21e-5	& 1.56e-4	& 2.72e-5	& 2.06e-4	& 4.44e-5	\\
10	& 2.58e-5	& 3.08e-6	& 6.48e-5	& 7.00e-6	& 9.13e-5	& 1.20e-5	\\
11	& 9.69e-6	& 6.20e-7	& 1.87e-5	& 1.73e-6	& 2.75e-5	& 3.25e-6	\\
12	& 1.67e-6	& 1.60e-7	& 4.80e-6	& 4.73e-7	& 8.46e-6	& 9.10e-7	\\
13	& 1.20e-6	& 3.61e-8	& 3.10e-6	& 1.24e-7	& 4.64e-6	& 2.60e-7	\\
14	& 2.61e-7	& 7.96e-9	& 1.67e-6	& 3.11e-8	& 2.61e-6	& 7.20e-8 \\
15	& 1.73e-7	& 1.76e-9	& 1.40e-6	& 8.11e-9	& 1.73e-6	& 2.01e-8 \\
\hline
\end{tabular}
\label{tb:6}}
\end{center}
\end{table}

\begin{table}
\begin{center}
{\scriptsize
\caption{Values of $\mathrm{rmse}(f;P_{2^m})$ for our constructed point sets, the competitors and interlaced scrambled polynomial lattice point sets with two choices of $s=1,2$.}
\begin{tabular}{l||lll|llll}
\hline
$m$ & \multicolumn{3}{|c|}{$s=1$} & \multicolumn{4}{|c}{$s=2$}  \\ \hline
  & Sobol' & PLPS-sc & PLPS & Sobol' & N-X & PLPS-sc & PLPS \\ \hline
4	& 1.72e-4	& 4.53e-5	& 8.09e-5	& 1.86e-4	& 1.05e-4	& 1.11e-4	& 1.29e-4	\\
5	& 4.31e-5	& 9.38e-6	& 2.09e-5	& 3.86e-5	& 2.15e-5	& 2.45e-5	& 2.99e-5	\\
6	& 1.02e-5	& 1.50e-6	& 5.90e-6	& 1.01e-5	& 1.09e-5	& 8.51e-6	& 6.70e-6	\\
7	& 2.92e-6	& 2.91e-7	& 1.30e-6	& 2.34e-6	& 1.39e-6	& 1.31e-6	& 1.20e-6	\\
8	& 6.82e-7	& 5.08e-8	& 3.39e-7	& 5.57e-7	& 4.69e-7	& 5.07e-7	& 4.42e-7	\\
9	& 1.59e-7	& 1.31e-8	& 8.50e-8	& 1.42e-7	& 6.85e-8	& 1.84e-7	& 1.92e-7	\\
10	& 3.77e-8	& 2.25e-9	& 2.07e-8	& 4.54e-8	& 1.96e-8	& 2.86e-8	& 1.90e-8	\\
11	& 1.00e-8	& 2.96e-10	& 5.03e-9	& 9.73e-9	& 5.02e-9	& 1.32e-8	& 6.39e-9	\\
12	& 2.37e-9	& 6.35e-11	& 1.23e-9	& 2.19e-9	& 9.90e-10	& 1.31e-9	& 1.03e-9	\\
13	& 5.97e-10	& 1.22e-11	& 3.10e-10	& 4.94e-10	& 3.03e-10	& 5.51e-10	& 2.75e-10	\\
14	& 1.60e-10	& 2.71e-12	& 7.32e-11	& 1.27e-10	& 1.11e-10	& 9.96e-11	& 6.82e-11	\\
15	& 3.43e-11	& 4.75e-13	& 1.81e-11	& 4.39e-11	& 1.54e-11	& 5.32e-11	& 2.00e-11	\\
\hline
\end{tabular}
\label{tb:7}}
\end{center}
\end{table}

\begin{table}
\begin{center}
{\scriptsize
\caption{Values of $\mathrm{rmse}(f;P_{2^m})$ for our constructed point sets, the competitors and interlaced scrambled polynomial lattice point sets with two choices of $s=5,10$.}
\begin{tabular}{l||llll|lll}
\hline
$m$ & \multicolumn{4}{|c|}{$s=5$} & \multicolumn{3}{|c}{$s=10$} \\ \hline
  & Sobol' & N-X & PLPS-sc & PLPS & Sobol' & PLPS-sc & PLPS \\ \hline
4	& 3.15e-4	& 1.63e-3	& 1.10e-4	& 1.29e-4	& 2.71e-4	& 9.61e-5	& 9.88e-5	\\
5	& 3.84e-5	& 1.63e-3	& 3.04e-5	& 3.18e-5	& 6.03e-5	& 2.72e-5	& 3.42e-5	\\
6	& 1.80e-5	& 2.16e-5	& 1.21e-5	& 1.05e-5	& 1.30e-5	& 1.14e-5	& 1.15e-5	\\
7	& 9.39e-6	& 7.26e-6	& 5.59e-6	& 3.25e-6	& 9.94e-6	& 4.46e-6	& 3.94e-6	\\
8	& 5.19e-6	& 1.38e-6	& 1.55e-6	& 1.24e-6	& 4.94e-6	& 1.27e-6	& 1.12e-6	\\
9	& 7.33e-7	& 3.79e-7	& 3.79e-7	& 2.67e-7	& 7.17e-7	& 3.74e-7	& 3.44e-7	\\
10	& 1.77e-7	& 2.73e-7	& 1.25e-7	& 6.59e-8	& 2.03e-7	& 2.16e-7	& 2.17e-7	\\
11	& 1.60e-7	& 6.72e-8	& 3.12e-8	& 6.52e-8	& 1.62e-7	& 4.21e-8	& 5.75e-8	\\
12	& 1.44e-8	& 3.65e-8	& 1.49e-8	& 9.92e-9	& 3.86e-8	& 2.56e-8	& 9.99e-9	\\
13	& 7.94e-9	& 3.94e-9	& 5.04e-9	& 4.02e-9	& 1.84e-8	& 7.38e-9	& 4.75e-9	\\
14	& 6.06e-10	& 8.13e-10	& 1.16e-9	& 8.30e-10	& 2.31e-9	& 2.20e-9	& 1.34e-9	\\
15	& 1.61e-10	& 1.84e-10	& 3.59e-10	& 1.42e-10	& 1.55e-9	& 9.91e-10	& 1.50e-9	\\
\hline
\end{tabular}
\label{tb:8}}
\end{center}
\end{table}

\begin{table}
\begin{center}
{\scriptsize
\caption{Values of $\mathrm{rmse}(f;P_{2^m})$ for our constructed point sets, the competitors and interlaced scrambled polynomial lattice point sets with two choices of $s=20,50$.}
\begin{tabular}{l||lll|lll}
\hline
$m$ & \multicolumn{3}{|c|}{$s=20$} & \multicolumn{3}{|c}{$s=50$} \\ \hline
  & Sobol' & PLPS-sc & PLPS & Sobol' & PLPS-sc & PLPS \\ \hline
4	& 2.59e-4	& 8.09e-5	& 1.12e-4	& 2.27e-4	& 9.40e-5	& 9.96e-5	\\
5	& 5.75e-5	& 2.36e-5	& 3.52e-5	& 5.64e-5	& 2.98e-5	& 2.86e-5	\\
6	& 1.45e-5	& 1.14e-5	& 1.11e-5	& 1.16e-5	& 1.13e-5	& 8.75e-6	\\
7	& 9.18e-6	& 4.68e-6	& 3.87e-6	& 9.31e-6	& 4.25e-6	& 4.09e-6	\\
8	& 4.58e-6	& 1.50e-6	& 1.21e-6	& 5.22e-6	& 1.10e-6	& 1.31e-6	\\
9	& 9.21e-7	& 3.92e-7	& 3.15e-7	& 1.00e-6	& 4.04e-7	& 3.26e-7	\\
10	& 2.67e-7	& 1.56e-7	& 1.98e-7	& 2.83e-7	& 2.08e-7	& 1.87e-7	\\
11	& 1.54e-7	& 4.55e-8	& 6.33e-8	& 1.52e-7	& 4.60e-8	& 6.31e-8	\\
12	& 5.11e-8	& 2.49e-8	& 1.15e-8	& 7.40e-8	& 2.77e-8	& 1.20e-8	\\
13	& 2.39e-8	& 6.15e-9	& 4.93e-9	& 3.06e-8	& 5.40e-9	& 6.39e-9	\\
14	& 1.02e-8	& 2.85e-9	& 2.12e-9	& 1.21e-8	& 2.15e-9	& 2.36e-9	\\
15	& 1.09e-8	& 9.90e-10	& 1.47e-9	& 1.13e-8	& 1.06e-9	& 1.37e-9	\\
\hline
\end{tabular}
\label{tb:9}}
\end{center}
\end{table}

\end{document}